\documentclass[11pt,a4paper]{article}
\usepackage{amssymb}
\usepackage{amsmath}
\usepackage{latexsym}
\usepackage{amsthm}
\usepackage{dsfont}
\usepackage{graphicx}
\usepackage{caption}
\usepackage{subcaption}

\newtheorem{thm}{Theorem}
\newtheorem{lem}[thm]{Lemma}
\theoremstyle{definition}
\newtheorem*{remark}{Remark}

\usepackage{xpatch}
\xpatchcmd{\proof}{\itshape}{\normalfont\proofnameformat}{}{}
\newcommand{\proofnameformat}{}

\begin{document}

\renewcommand{\proofnameformat}{\bfseries}

\begin{center}
{\Large\textbf{Limit laws of maximal Birkhoff sums for circle rotations via quantum modular forms}}

\vspace{10mm}

\textbf{Bence Borda}

{\footnotesize Graz University of Technology

Steyrergasse 30, 8010 Graz, Austria

Email: \texttt{borda@math.tugraz.at}}

\vspace{5mm}

{\footnotesize \textbf{Keywords:} continued fraction, Gauss map, Ostrowski expansion, Farey fraction,

quadratic irrational, Kashaev invariant, Sudler product}

\end{center}

\vspace{5mm}

\begin{abstract}
In this paper, we show how quantum modular forms naturally arise in the ergodic theory of circle rotations. Working with the classical Birkhoff sum $S_N(\alpha)=\sum_{n=1}^N (\{ n \alpha \}-1/2)$, we prove that the maximum and the minimum as well as certain exponential moments of $S_N(r)$ as functions of $r \in \mathbb{Q}$ satisfy a direct analogue of Zagier's continuity conjecture, originally stated for a quantum invariant of the figure-eight knot. As a corollary, we find the limit distribution of $\max_{0 \le N<M} S_N(\alpha)$ and $\min_{0 \le N<M} S_N(\alpha)$ with a random $\alpha \in [0,1]$.
\end{abstract}

\section{Introduction}

The main goal of this paper is to introduce methods originally developed in connection with Zagier's quantum modular forms \cite{ZA} to the ergodic theory of circle rotations. We demonstrate the power of these tools by considering the classical Birkhoff sum $S_N (\alpha) = \sum_{n=1}^N (\{ n \alpha \} - 1/2)$, where $\{ \cdot \}$ denotes the fractional part function. The history of the sum $S_N(\alpha)$ goes back a hundred years to Hardy and Littlewood \cite{HL1,HL2}, Hecke \cite{HE} and Ostrowski \cite{OS}, with the original motivation coming from Diophantine approximation, lattice point counting in triangles and analytic number theory. We have $S_N(\alpha)=o(N)$ for any irrational $\alpha$, but the precise behavior is rather delicate and depends on the Diophantine properties of $\alpha$. It is enough to consider $\alpha \in [0,1]$, and we shall focus on the case of a randomly chosen $\alpha$.

Throughout, $X \sim \mu$ denotes the fact that a random variable $X$ has distribution $\mu$, $\mu \otimes \nu$ denotes the product measure of $\mu$ and $\nu$, and $\overset{d}{\to}$ denotes convergence in distribution. The standard stable law of stability parameter $1$ and skewness parameter $\pm 1$, denoted by $\mathrm{Stab}(1,\pm 1)$, is the law with characteristic function $\exp (-|x| (1 \pm i \frac{2}{\pi} \mathrm{sgn} (x) \log |x|))$. The standard stable law of stability parameter $1$ and skewness parameter $0$ is in fact the standard Cauchy distribution with characteristic function $\exp (-|x|)$ and density function $1/(\pi (1+x^2))$, and will be denoted simply by ``Cauchy''.

The first distributional result is due to Kesten \cite{KE}, who proved that if $(\alpha,\beta) \sim \mathrm{Unif}([0,1]^2)$, then
\begin{equation}\label{kesten}
\frac{\sum_{n=1}^N \left( \{ n \alpha + \beta \} -1/2 \right)}{\sigma \log N} \overset{d}{\to} \mathrm{Cauchy}
\end{equation}
as $N \to \infty$, with an explicit constant $\sigma>0$. Note that in addition to $\alpha$, the starting point $\beta$ of the orbit is also chosen randomly, independently of $\alpha$. Whether a similar limit law holds for a fixed value of $\beta$ is still open. Dolgopyat and Sarig \cite{DS1} showed, however, that for any fixed $\beta \in \mathbb{R}$ and $(\alpha, N) \sim \mathrm{Unif} ([0,1] \times \{ 1,2,\ldots, M \})$, the limit law \eqref{kesten} holds as $M \to \infty$ with the different constant $\sigma=\frac{1}{3 \pi \sqrt{3}}$. Let us also mention a theorem of Beck \cite{BE} concerning $\beta=0$, a fixed quadratic irrational $\alpha$ and $N \sim \mathrm{Unif}(\{ 1,2,\ldots, M \})$, in which case $(S_N(\alpha) - c_1 \log N)/(c_2 \sqrt{\log N})$ converges in distribution to the standard Gaussian with suitable constants $c_1 \in \mathbb{R}$ and $c_2>0$ depending on $\alpha$.

In this paper, we work with $S_N (\alpha) = \sum_{n=1}^N (\{ n \alpha \} - 1/2)$ with the fixed starting point $\beta =0$, and instead of choosing $N$ randomly, we consider the extreme values $\max_{0 \le N<M} S_N(\alpha)$ and $\min_{0 \le N<M} S_N (\alpha)$ as well as certain exponential moments of the values $S_N(\alpha)$, $0 \le N<M$. Our main distributional result is a limit law for the joint distribution of the maximum and the minimum.
\begin{thm}\label{maxmintheorem} Let $\alpha \sim \mu$ with a Borel probability measure $\mu$ on $[0,1]$ which is absolutely continuous with respect to the Lebesgue measure. Then
\[ \left( \frac{\displaystyle{\max_{0 \le N<M} S_N (\alpha) - E_M}}{\sigma_M}, \frac{\displaystyle{\min_{0 \le N<M} S_N (\alpha) + E_M}}{\sigma_M} \right) \overset{d}{\to} \mathrm{Stab}(1,1) \otimes \mathrm{Stab}(1,-1) \qquad \textrm{as } M \to \infty, \]
where $E_M = \frac{3}{4 \pi^2} \log M \log \log M + D_{\infty} \log M$ with some constant $D_{\infty} \in \mathbb{R}$, and $\sigma_M = \frac{3}{8 \pi} \log M$.
\end{thm}
\noindent In particular,
\[ \frac{\displaystyle{\max_{0 \le N <M} S_N(\alpha) - E_M}}{\sigma_M} \overset{d}{\to} \mathrm{Stab} (1,1) \qquad \textrm{and} \qquad \frac{\displaystyle{\min_{0 \le N <M} S_N(\alpha) + E_M}}{\sigma_M} \overset{d}{\to} \mathrm{Stab} (1,-1) . \]
The fact that the limit distribution in Theorem \ref{maxmintheorem} is a product measure means that the maximum and the minimum of $S_N(\alpha)$ are asymptotically independent. The formulation as a joint limit law has the advantage that we immediately obtain limit laws for quantities such as $\max - \min$ (the diameter of the range of $S_N (\alpha)$, $0 \le N<M$), and for $(\max +\min)/2$ (the center of the range) as well:
\[ \frac{\displaystyle{\max_{0 \le N <M} S_N(\alpha) - \min_{0 \le N<M} S_N(\alpha) - B_M}}{2\sigma_M} \overset{d}{\to} \mathrm{Stab} (1,1), \quad \frac{\displaystyle{\max_{0 \le N<M} S_N(\alpha) + \min_{0 \le N<M} S_N(\alpha)}}{2 \sigma_M} \overset{d}{\to} \mathrm{Cauchy} \]
with $B_M=2E_M+\frac{4}{\pi} (\log 2) \sigma_M$. Indeed, if $X,Y \sim \mathrm{Stab}(1,1)$ are independent random variables, then $-X \sim \mathrm{Stab}(1,-1)$, $\frac{X+Y}{2}-\frac{2}{\pi} \log 2 \sim \mathrm{Stab}(1,1)$ and $\frac{X-Y}{2} \sim \mathrm{Cauchy}$, as can be easily seen from the characteristic functions. Theorem \ref{maxmintheorem} similarly implies that
\[ \frac{\displaystyle{\max_{0 \le N<M} |S_N(\alpha)| -E_M}}{\sigma_M} \overset{d}{\to} \max \{ X,Y \} \qquad \textrm{as } M \to \infty . \]
The cumulative distribution function of $\max \{ X,Y \}$ is simply the square of that of $\mathrm{Stab}(1,1)$.

Limit laws of Birkhoff sums for circle rotations $\sum_{n=1}^N f(n \alpha + \beta)$ with some of the parameters $N,\alpha,\beta$ chosen randomly have also been established for other $1$-periodic functions $f$, such as the indicator of a subinterval of $[0,1]$ extended with period $1$, or smooth functions with a logarithmic or power singularity. We refer to \cite{DF} for an exhaustive survey. In an upcoming paper we will prove similar limit laws for the maximum and the minimum of $\sum_{n=1}^N f(n \alpha)$ with $f$ the indicator of a subinterval of $[0,1]$ extended with period $1$, using methods unrelated to the present paper.

Our approach relies on continued fractions and Ostrowski's explicit formula for $S_N(\alpha)$, see Lemma \ref{ostrowskilemma} below. We will actually work with $S_N(r)$ with rational $r$ instead of an irrational $\alpha$, and eventually let $r$ be a suitable best rational approximation to a random $\alpha$. As the main ingredient in the proof of our limit laws, we will show that while $\max_{0 \le N<q} S_N(r)$ and $\min_{0 \le N<q} S_N(r)$ are rather complicated as functions of the variable $r \in (0,1) \cap \mathbb{Q}$, the functions
\[ h_{\infty} (r) = \max_{0 \le N <q} S_N(r) - \max_{0 \le N <q'} S_N(T^2 r) \qquad \textrm{and} \qquad h_{- \infty} (r) = \min_{0 \le N <q} S_N(r) - \min_{0 \le N <q'} S_N(T^2 r) \]
have better analytic properties in the sense that they can be extended to almost everywhere continuous functions on $[0,1]$; see Figures \ref{figure1} and \ref{figure2} below. Here $T^2$ is the second iterate of the Gauss map, and $q$ resp.\ $q'$ denotes the denominator of $r$ resp.\ $T^2 r$ in their reduced forms. This makes the functions $\max_{0 \le N <q} S_N(r)$ and $\min_{0 \le N<q} S_N(r)$ close relatives of Zagier's quantum modular forms, an observation we believe to be of independent interest.

We argue that $S_N(\alpha)$ shows a close similarity to $\tilde{S}_N(\alpha) = \sum_{n=1}^N \log |2 \sin (\pi n \alpha)|$, the Birkhoff sum with the $1$-periodic function $\log |2 \sin (\pi x)|$ having logarithmic singularities at integers. This similarity is not surprising considering that $\tilde{S}_N(\alpha)$ and $\pi S_N(\alpha)$ are the real and the imaginary part of the complex-valued Birkhoff sum $\sum_{n=1}^N \log (1-e^{2 \pi i n \alpha})$, defined with the principal branch of the logarithm. Note that $e^{\tilde{S}_N(\alpha)} = \prod_{n=1}^N |1-e^{2 \pi i n \alpha}|$ is the so-called Sudler product, a classical object in its own right introduced by Sudler \cite{SU} and Erd\H{o}s and Szekeres \cite{ESZ}. Confirming a conjecture of Zagier, in a recent paper Aistleitner and the author \cite{AB1} proved that while $\max_{0 \le N<q} \tilde{S}_N(r)$ and $\min_{0 \le N<q} \tilde{S}_N(r)$ exhibit complicated behavior, the functions
\[ \tilde{h}_{\infty} (r) = \max_{0 \le N <q} \tilde{S}_N(r) - \max_{0 \le N <q'} \tilde{S}_N(T r) \qquad \textrm{and} \qquad \tilde{h}_{- \infty} (r) = \min_{0 \le N <q} \tilde{S}_N(r) - \min_{0 \le N <q'} \tilde{S}_N(T r) \]
can be extended to almost everywhere continuous functions on $[0,1]$. The results of the present paper suggest that such behavior is more prevalent than the original scope of Zagier's continuity conjecture.

It is rather surprising that the functions $h_{\pm \infty}$ and $\tilde{h}_{\pm \infty}$ with such a pathological behavior hold the key to limit laws such as Theorem \ref{maxmintheorem}. Improving our earlier result \cite[Theorem 10]{BO}, in this paper we also prove that if $\alpha \sim \mu$ with an absolutely continuous probability measure $\mu$ on $[0,1]$, then
\begin{equation}\label{maxtildelimitlaw}
\frac{\displaystyle{\max_{0 \le N<M} \tilde{S}_N(\alpha) - \tilde{E}_M}}{\tilde{\sigma}_M} \overset{d}{\to} \mathrm{Stab}(1,1) \qquad \textrm{as } M \to \infty ,
\end{equation}
where $\tilde{E}_M=\frac{3 \mathrm{Vol}(4_1)}{\pi^3} \log M \log \log M + \tilde{D}_{\infty} \log M$ and $\tilde{\sigma}_M= \frac{3\mathrm{Vol}(4_1)}{2 \pi^2} \log M$, with
\[ \mathrm{Vol} (4_1) = 4 \pi \int_0^{5/6} \log |2 \sin (\pi x)| \, \mathrm{d}x = 2.02988\ldots \]
denoting the hyperbolic volume of the complement of the figure-eight knot (see Section \ref{quantumsection}) and some constant $\tilde{D}_{\infty} \in \mathbb{R}$. The maximum and the minimum now determine each other via the relation
\[ \max_{0 \le N <M} \tilde{S}_N(\alpha) + \min_{0 \le N<M} \tilde{S}_N(\alpha) = \log M +o(\log M) \qquad \textrm{in $\mu$-measure} , \]
which easily follows from \cite[Eq.\ (17)]{AB2}. This immediately yields a limit law for $\min_{0\le N <M} \tilde{S}_N(\alpha)$ as well, and shows that in contrast to Theorem \ref{maxmintheorem}, the joint distribution of
\[ \left( \frac{\displaystyle{\max_{0 \le N<M} \tilde{S}_N(\alpha) - \tilde{E}_M}}{\tilde{\sigma}_M}, \frac{\displaystyle{\min_{0 \le N<M} \tilde{S}_N(\alpha) + \tilde{E}_M}}{\tilde{\sigma}_M} \right) \]
converges to a probability measure supported on a straight line in $\mathbb{R}^2$ instead of a product measure. The difference in the definition of $h_{\pm \infty}$ and $\tilde{h}_{\pm \infty}$ (second vs.\ first iterate of the Gauss map) and in the joint behavior of the maximum and the minimum (asymptotically independent vs.\ asymptotically deterministic) ultimately boils down to the fact that $S_N(\alpha)$ is odd, whereas $\tilde{S}_N(\alpha)$ is even in the variable $\alpha$. See also \cite{HA,LU} for the asymptotics of $\tilde{S}_N(\alpha)$ at a.e.\ $\alpha$.

In contrast to random reals, for a badly approximable irrational $\alpha$ we have $S_N (\alpha)=O(\log N)$, and this is sharp since
\begin{equation}\label{ostrowskilimsupliminf}
\limsup_{N \to \infty} \frac{S_N(\alpha)}{\log N} >0 \qquad \textrm{and} \qquad \liminf_{N \to \infty} \frac{S_N(\alpha)}{\log N}<0,
\end{equation}
as shown by Ostrowski \cite{OS}. For a quadratic irrational $\alpha$, we can say more: general results of Schoissengeier \cite{SCH} on $S_N(\alpha)$ immediately imply that
\begin{equation}\label{quadraticmaxmin}
\max_{0\le N<M} S_N (\alpha) = C_{\infty}(\alpha) \log M +O(1) \qquad \textrm{and} \qquad \min_{0 \le N<M} S_N(\alpha) = C_{-\infty}(\alpha) \log M +O(1)
\end{equation}
with some explicitly computable constants $C_{\infty}(\alpha)>0$ and $C_{-\infty}(\alpha)<0$, and implied constants depending only on $\alpha$. Note that $C_{\infty}(\alpha)$ resp.\ $C_{-\infty}(\alpha)$ is the value of the limsup resp.\ liminf in \eqref{ostrowskilimsupliminf}. For example, we have
\[ C_{\pm \infty} (\sqrt{2}) = \pm \frac{1}{8 \log (1+\sqrt{2})}, \quad C_{\infty}(\sqrt{3}) = \frac{1}{4 \log (2+\sqrt{3})}, \quad C_{-\infty}(\sqrt{3}) = -\frac{1}{12 \log (2+\sqrt{3})} . \]

Similar results hold for $\tilde{S}_N(\alpha)$. For all badly approximable irrational $\alpha$ we have $\tilde{S}_N(\alpha)=O(\log N)$, and this is sharp since $\limsup_{N \to \infty} \tilde{S}_N(\alpha)/\log N \ge 1$ for all (not necessarily badly approximable) irrationals \cite{LU}. For a quadratic irrational $\alpha$, we similarly have \cite{AB2}
\[ \max_{0\le N<M} \tilde{S}_N (\alpha) = \tilde{C}_{\infty}(\alpha) \log M +O(1) \qquad \textrm{and} \qquad \min_{0 \le N<M} \tilde{S}_N(\alpha) = \tilde{C}_{-\infty}(\alpha) \log M +O(1) . \]
Here the constants $\tilde{C}_{\infty}(\alpha) \ge 1$ and $\tilde{C}_{-\infty}(\alpha) \le 0$ are related by $\tilde{C}_{\infty}(\alpha)+\tilde{C}_{-\infty}(\alpha)=1$, but their explicit value is known only for a few simple quadratic irrationals such as the golden mean or $\sqrt{2}$ (in both cases $\tilde{C}_{\infty}=1$ and $\tilde{C}_{-\infty}=0$). Thus, once again, the maximum and the minimum of $\tilde{S}_N(\alpha)$ determine each other, unlike those of $S_N(\alpha)$ for which the constants $C_{\infty}(\alpha)$ and $C_{-\infty}(\alpha)$ do not satisfy a simple relation. We refer to our earlier paper \cite{BO} for a central limit theorem for the joint distribution of $(S_N(\alpha), \tilde{S}_N(\alpha))$ with a fixed quadratic irrational $\alpha$ and $N \sim \mathrm{Unif}(\{ 1,2,\ldots, M \})$.

We elaborate on the connection to quantum modular forms, and state our main related results in Section \ref{quantumsection}. The main limit laws, including more general forms of Theorem \ref{maxmintheorem} and formula \eqref{maxtildelimitlaw} together with analogue results for random rationals are stated in Section \ref{limilawsection}. The proofs are given in Sections \ref{proofhpsection}, \ref{proofquadraticsection} and \ref{prooflimitlawsection}.

\section{Connections to quantum modular forms}\label{quantumsection}

A quantum modular form is a real- or complex-valued function $f$ defined on $\mathbb{P}^1 (\mathbb{Q}) = \mathbb{Q} \cup \{ \infty \}$ (except perhaps at finitely many points) which satisfies a certain approximate modularity relation under the action of $\mathrm{SL}(2,\mathbb{Z})$ with fractional linear transformations on $\mathbb{P}^1(\mathbb{Q})$. Instead of stipulating $f(\gamma r)=f(r)$ for any $\gamma \in \mathrm{SL}(2, \mathbb{Z})$ (true modularity), the functions $h_{\gamma}(r)=f(\gamma r) - f(r)$ are required, roughly speaking, to enjoy better continuity/analyticity properties than $f$ itself in the real topology on $\mathbb{P}^1(\mathbb{Q})$ (approximate modularity). Most known examples of quantum modular forms come from algebraic topology or analytic number theory.

Given a parameter $-\infty \le p \le \infty$, $p \neq 0$ and a rational number $r$ whose denominator in its reduced form is $q$, define
\[ \tilde{J}_p (r) = \left( \sum_{N=1}^{q-1} \prod_{n=1}^N |1-e^{2 \pi i n r}|^p \right)^{1/p} \qquad p \neq \pm \infty, 0, \]
and
\[ \tilde{J}_{\infty} (r) = \max_{0 \le N<q} \prod_{n=1}^N |1-e^{2 \pi i n r}|, \qquad \tilde{J}_{-\infty} (r) = \min_{0 \le N<q} \prod_{n=1}^N |1-e^{2 \pi i n r}| , \]
where $\prod_{n=1}^N |1-e^{2 \pi i n r}|$ is the Sudler product. The function $\tilde{J}_p(r)$ is $1$-periodic and even in the variable $r$, and by \cite[Proposition 2]{AB2} it also satisfies the identity $\tilde{J}_{-p}(r)=q/\tilde{J}_p(r)$.

The original motivation came from algebraic topology, as $\tilde{J}_2^2$ is (an extension of) the so-called Kashaev invariant of the figure-eight knot $4_1$. The asymptotics along the sequence of rationals $r=1/q$, $q \in \mathbb{N}$ is
\begin{equation}\label{volumeconjecture}
\log \tilde{J}_2 (1/q) = \frac{\mathrm{Vol}(4_1)}{4 \pi} q + \frac{3}{4} \log q - \frac{1}{8} \log 3 +o(1) \qquad \textrm{as } q \to \infty,
\end{equation}
where $\mathrm{Vol}(4_1)$ is the hyperbolic volume of the complement of the figure-eight knot \cite{AH}. A similar asymptotic result for the Kashaev invariant of general hyperbolic knots is known as the volume conjecture, with a full asymptotic expansion in $q$ predicted by the arithmeticity conjecture. Both conjectures have been solved for certain simple hyperbolic knots such as the figure-eight knot, but are open in general.

Calling $\tilde{J}_2$ ``the most mysterious and in many ways the most interesting'' example of a quantum modular form, Zagier \cite{ZA} formulated several conjectures about its behavior under the action of $\mathrm{SL}(2, \mathbb{Z})$ on its argument by fractional linear transformations, including a far-reaching generalization of \eqref{volumeconjecture} known as the modularity conjecture. Zagier's modularity conjecture has a more general form which applies to all hyperbolic knots, but it has only been solved for certain simple knots such as the figure-eight knot \cite{BD2}, and remains open in general. We refer to \cite{BD2} for further discussion on the arithmetic properties of quantum invariants of hyperbolic knots.

Since the fractional linear maps $r \mapsto r+1$ and $r \mapsto -1/r$ generate the full modular group, and the first of these transformations acts trivially on the argument of $\tilde{J}_p(r)$, the function
\[ \tilde{h}_p (r) = \log \frac{\tilde{J}_p(r)}{\tilde{J}_p (-1/r)}, \qquad r \in \mathbb{Q} \backslash \{0 \} \]
is the key to understanding the action of $\mathrm{SL}(2, \mathbb{Z})$. Observe that $\tilde{h}_{-p} (r) = -\tilde{h}_p(r)$, hence it is enough to consider $p>0$. Numerical evidence presented by Zagier suggests that $\tilde{h}_2$ is continuous but not differentiable at every irrational, and that it has a jump discontinuity at every rational but is smooth as we approach a rational from one side. The continuity of $\tilde{h}_2$ at all irrationals is now known as Zagier's continuity conjecture. Aistleitner and the author \cite{AB1} proved that $\tilde{h}_p$ can be extended to a function on $\mathbb{R}$ which is continuous at every irrational $\alpha=[a_0;a_1,a_2,\ldots]$ such that $\sup_{k \in \mathbb{N}}a_k = \infty$, thereby confirming Zagier's continuity conjecture almost everywhere. In the same paper it was further shown that
\begin{equation}\label{tildehpasymptotics}
\tilde{h}_p (r) = \frac{\mathrm{Vol}(4_1)}{4 \pi r} + O \left( 1+\log \frac{1}{r} \right) , \qquad r \in (0,1) \cap \mathbb{Q}
\end{equation}
with an implied constant depending only on $p$ (but it is uniform once $p$ is bounded away from $0$). Numerical experiments suggest that in fact
\[ \tilde{h}_p (r) = \frac{\mathrm{Vol}(4_1)}{4 \pi r} + \frac{p+1}{2p} \log \frac{1}{r} + O (1) , \qquad r \in (0,1) \cap \mathbb{Q} . \]
Note that in \cite{AB1} these results were stated only for $p=2$, but the proof works mutatis mutandis for all $0< p \le \infty$.

In this paper, we interpret $\tilde{J}_p$ as a natural quantity related to the Birkhoff sum $\tilde{S}_N (r) = \sum_{n=1}^N \log |2 \sin (\pi nr)|$, and $\tilde{h}_p$ as the key to understanding the action of the Gauss map $T$ on the argument of $\tilde{J}_p$. Recall that $T: [0,1) \to [0,1)$ is defined as $Tx=\{ 1/x \}$, $x \neq 0$ and $T0=0$, thus $\tilde{h}_p(r)=\log (\tilde{J}_p(r)/\tilde{J}_p (Tr))$. We show that the Birkhoff sum $S_N(r) = \sum_{n=1}^N (\{ n r \} -1/2)$ yields a function $J_p(r)$ which exhibits remarkable similarity to $\tilde{J}_p(r)$, thus demonstrating that quantum modular behavior can also naturally arise in ergodic theory. It would be very interesting to find further examples of Birkhoff sums, either for circle rotations or more general dynamical systems, with a similarly rich arithmetic structure.

Given a parameter $-\infty \le p \le \infty$, $p \neq 0$ and a rational number $r$ whose denominator in its reduced form is $q$, we thus define
\[ J_p (r) = \left( \sum_{N=0}^{q-1} e^{p S_N (r)} \right)^{1/p}, \qquad p \neq \pm \infty, 0, \]
and
\[ J_{\infty} (r) = \max_{0 \le N<q} e^{S_N (r)}, \qquad J_{-\infty} (r) = \min_{0 \le N<q} e^{S_N(r)} . \]
Note that these are perfect analogues of $\tilde{J}_p(r)$ with $S_N(r)$ playing the role of $\tilde{S}_N(r)$. Using the fact that $S_N(r)$ is $1$-periodic and odd in the variable $r$, we immediately observe the identities $J_p(r+1)=J_p(r)$ and $J_{-p}(r)=1/J_p (-r)$. In order to reveal the arithmetic structure of $J_p$, we introduce the function
\[ h_p(r) = \log \frac{J_p(r)}{J_p (T^2 r)}, \qquad r \in [0,1) \cap \mathbb{Q} , \]
where $T^2$ is the second iterate of the Gauss map.

\begin{figure}[ht]
\begin{center}
\includegraphics[width=0.6 \linewidth]{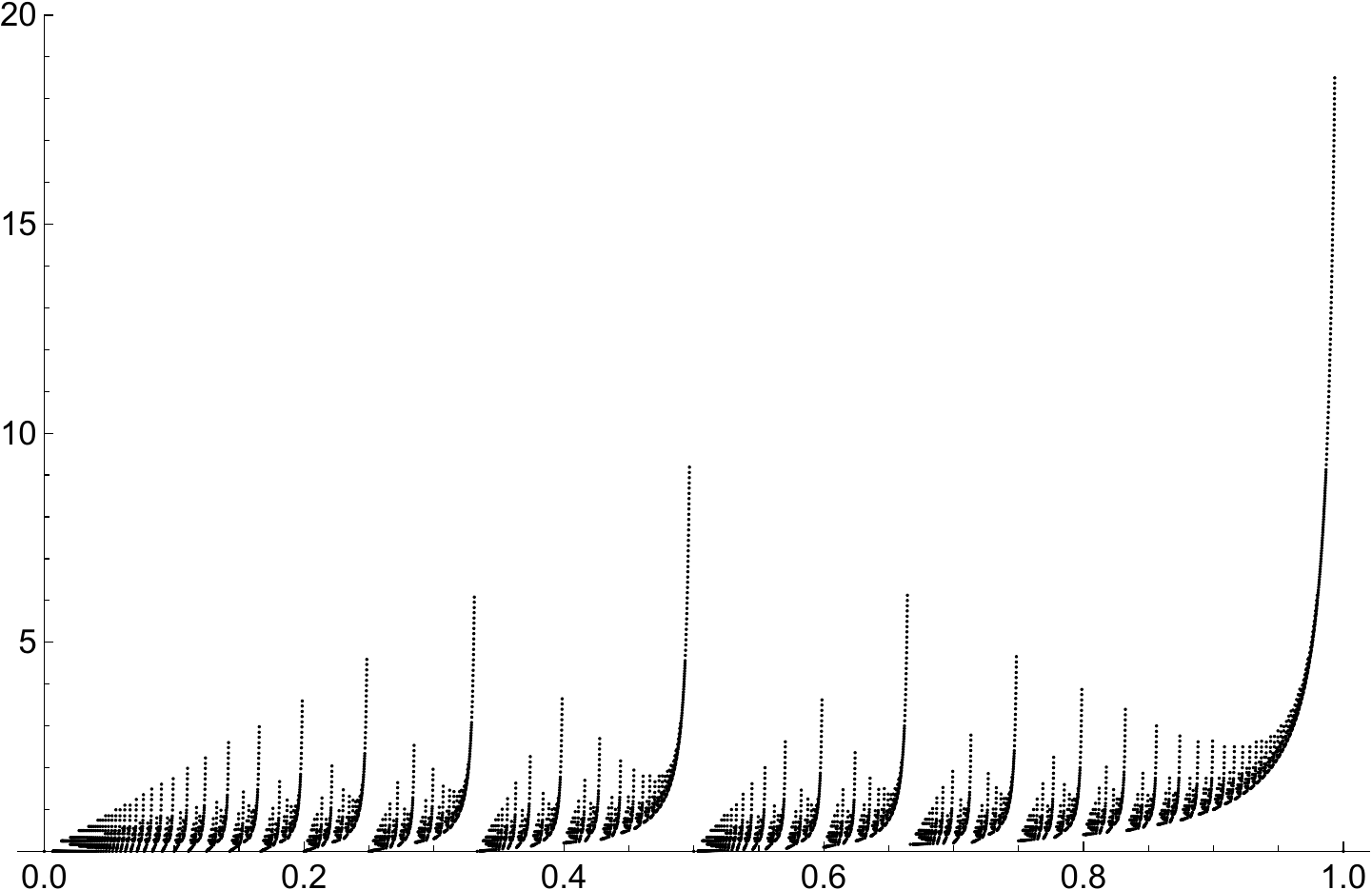}
\end{center}
\caption{The function $\log J_{\infty}(r)=\max_{0\le N<q} S_N(r)$ evaluated at all reduced rationals in $[0,1]$ with denominator at most $150$. The graph of $\log J_p(r)$ with $0<p<\infty$ looks very similar, whereas the graph of $\log J_{-p}(r)=-\log J_p(-r)$ is obtained by reflections.}
\label{figure1}
\end{figure}

The analogue of \eqref{volumeconjecture} for $J_p$ is completely straightforward. Indeed, for $r=1/q$, $q \in \mathbb{N}$, we have
\[ S_N(1/q) = \sum_{n=1}^N \left( \frac{n}{q} - \frac{1}{2} \right) = \frac{N(N+1-q)}{2q} , \qquad 0 \le N<q, \]
and it is an easy exercise to show that (cf.\ Lemma \ref{aplemma} below)
\begin{equation}\label{1/q}
h_p (1/q) = \log J_p (1/q) = \left\{ \begin{array}{ll} \frac{1}{p} \log \frac{2}{1-e^{-p/2}} +o(1) & \textrm{if } 0<p<\infty, \\ - \frac{q}{8} + \frac{1}{2p} \log \frac{2 \pi q}{|p|} + \frac{1}{4} +o(1) & \textrm{if } -\infty < p <0 \end{array} \right. \qquad \textrm{as } q \to \infty .
\end{equation}
Since $S_N(1/q)$, $0 \le N<q$ attains its maximum at $N=0,q-1$ and its minimum at $N=\lfloor \frac{q-1}{2} \rfloor, \lceil \frac{q-1}{2} \rceil$, for $p=\pm \infty$ we even have the explicit formulas
\[ h_{\infty} (1/q) = \log J_{\infty} (1/q) = 0 \qquad \textrm{and} \qquad h_{-\infty} (1/q) = \log J_{-\infty} (1/q) = - \frac{q}{8} + \frac{1}{4} - \frac{1}{8q} \mathds{1}_{\{ q \textrm{ odd} \}} . \]

As a direct analogue of \eqref{tildehpasymptotics}, we establish a far-reaching generalization of the asymptotics \eqref{1/q} to general rationals.
\begin{thm}\label{asymptoticstheorem} For any $-\infty \le p \le \infty$, $p \neq 0$ and any $r \in (0,1) \cap \mathbb{Q}$,
\[ h_p (r) = \left\{ \begin{array}{ll} \mathds{1}_{\{ Tr \neq 0 \}} \left( \frac{1}{8 Tr} + \frac{1}{2p} \log \frac{1}{Tr} \right) + O \left( \max \{ 1, \frac{1}{p} \log \frac{1}{p} \} \right) & \textrm{if } p>0, \\ - \frac{1}{8r} + \frac{1}{2p} \log \frac{1}{r} + O \left( \max \{ 1, \frac{1}{|p|} \log \frac{1}{|p|} \} \right) & \textrm{if } p<0 \end{array} \right. \]
with a universal implied constant.
\end{thm}
\noindent We can express Theorem \ref{asymptoticstheorem} in terms of the continued fraction expansion $r=[0;a_1,a_2,\ldots, a_L]$ of $r \in (0,1) \cap \mathbb{Q}$ as
\[ h_p (r) = \left\{ \begin{array}{ll} \frac{a_2}{8} + \frac{1}{2p} \log a_2 + O \left( \max \{ 1, \frac{1}{p} \log \frac{1}{p} \} \right) & \textrm{if } p>0, \\ - \frac{a_1}{8} + \frac{1}{2p} \log a_1 + O \left( \max \{ 1, \frac{1}{|p|} \log \frac{1}{|p|} \} \right) & \textrm{if } p<0 . \end{array} \right.  \]
\begin{remark} In all our results, it does not matter which of the two possible continued fraction expansions we choose for a rational number. In particular, to avoid the tedious case distinction between the length of the continued fraction being $L=1$ or $L \ge 2$, we consider the second partial quotient of $r=1/q=[0;q]=[0;q-1,1]$ (when $Tr=0$) to be well defined as $a_2=1$.
\end{remark}

\begin{figure}[ht]
\begin{center}
\begin{subfigure}{0.47\textwidth}
\includegraphics[width=1 \linewidth]{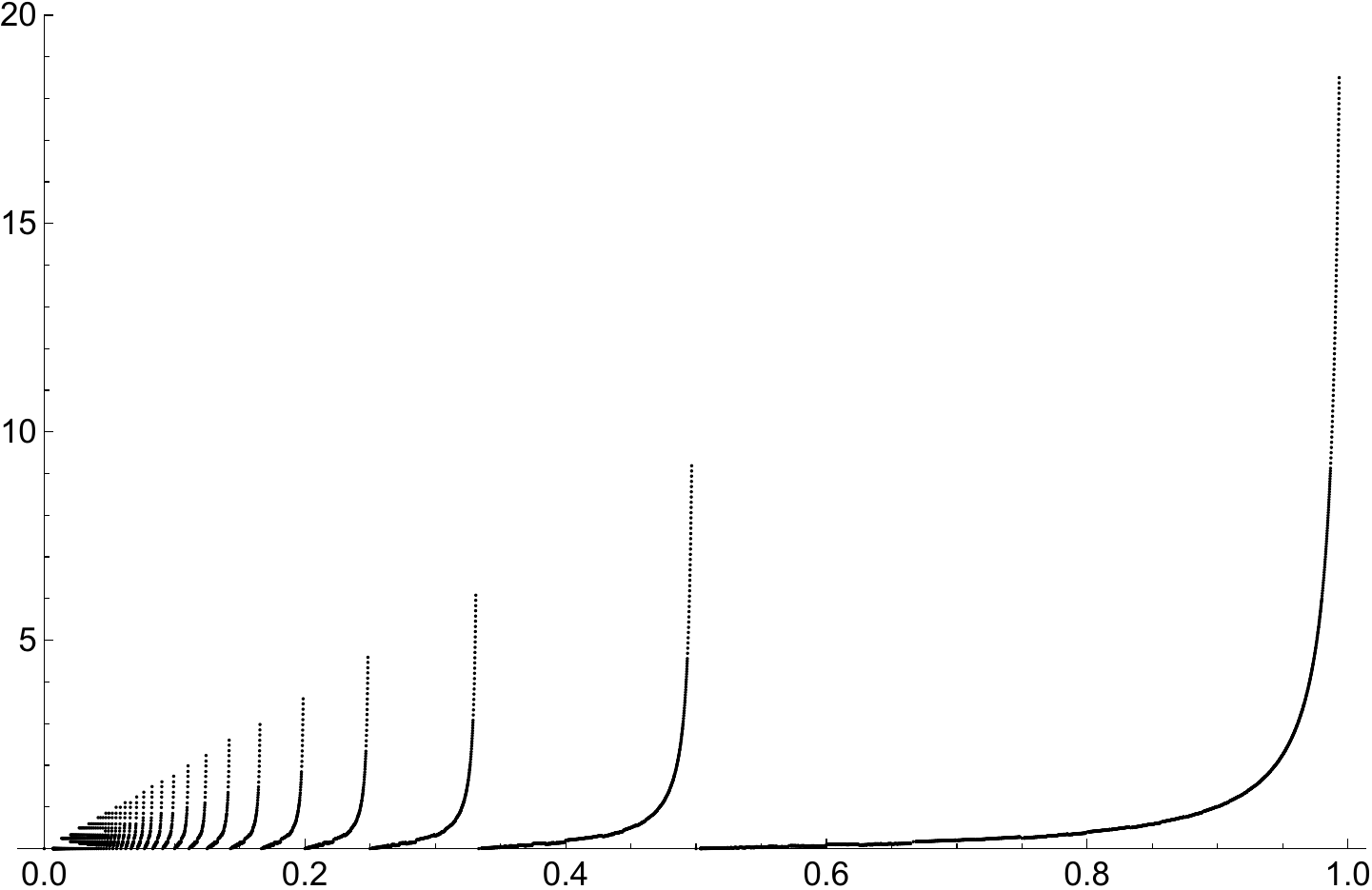}
\caption{$h_{\infty}(r)$}
\end{subfigure}
\hspace{5mm}
\begin{subfigure}{0.47\textwidth}
\includegraphics[width=1 \linewidth]{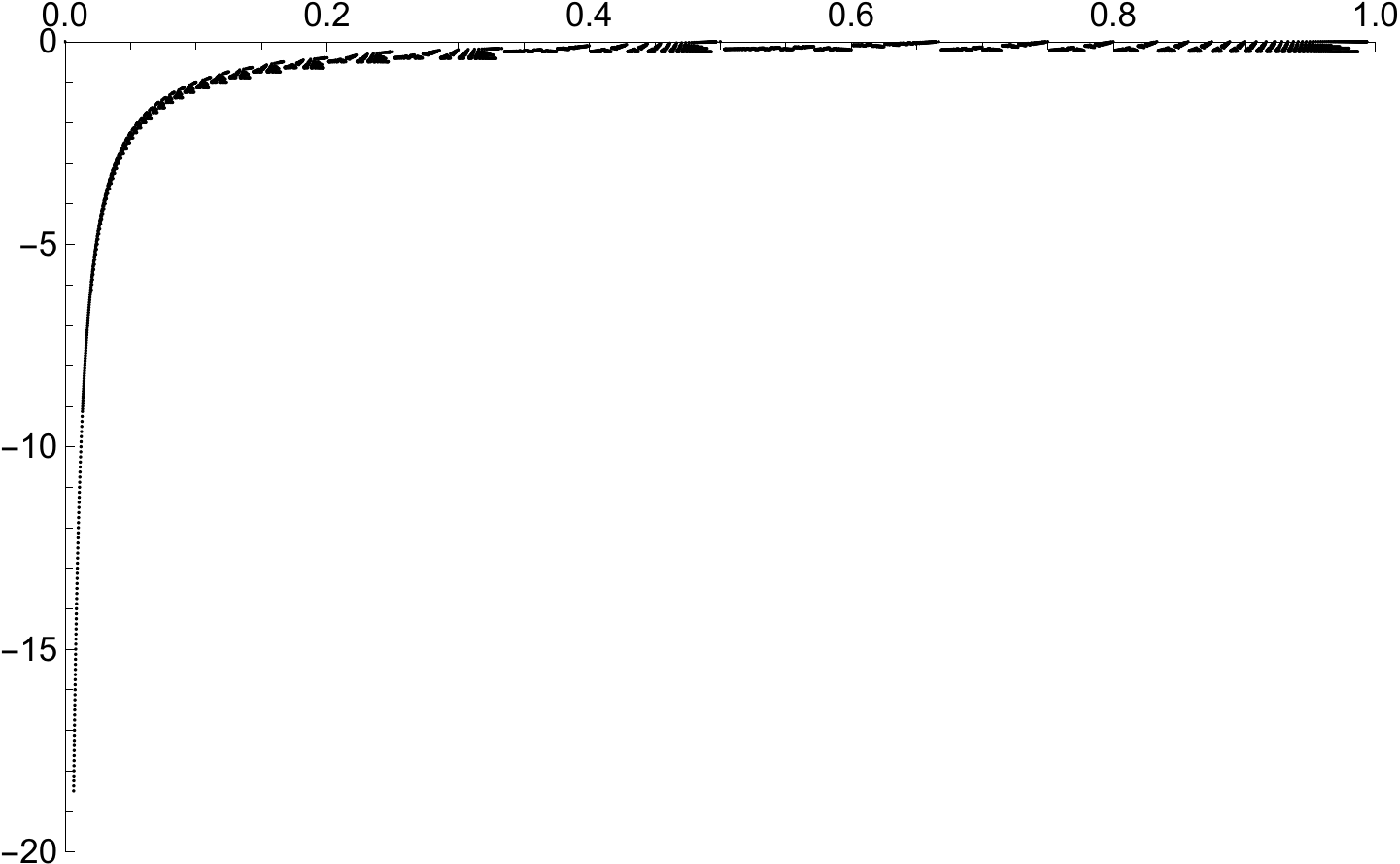}
\caption{$h_{-\infty}(r)$}
\end{subfigure}
\end{center}
\caption{The functions $h_{\pm \infty}(r)$ evaluated at all reduced rationals in $[0,1)$ with denominator at most $150$. The asymptotics $1/(8Tr)$ resp.\ $-1/(8r)$ in Theorem \ref{asymptoticstheorem} give a close fit to the graphs.}
\label{figure2}
\end{figure}

Our next result concerns the continuity of $h_p$ at irrationals, as an analogue of Zagier's continuity conjecture. For the sake of readability, from now on we use the notation
\begin{equation}\label{varepsilonp}
\varepsilon_p = \left\{ \begin{array}{ll} 2 & \textrm{if } p>0, \\ 1 & \textrm{if } p<0 . \end{array} \right.
\end{equation}
\begin{thm}\label{continuitytheorem} Let $-\infty \le p \le \infty$, $p \neq 0$, and let $\alpha \in (0,1)$ be an irrational whose continued fraction expansion $\alpha = [0;a_1,a_2,\ldots]$ satisfies $\sup_{k \in \mathbb{N}} a_{2k+\varepsilon_p}= \infty$. Then $\lim_{r \to \alpha} h_p(r)$ exists and is finite. In particular, $h_p$ can be extended to a function on $[0,1]$ which is continuous at every irrational $\alpha$ which satisfies $\sup_{k \in \mathbb{N}} a_{2k+\varepsilon_p}= \infty$.
\end{thm}
\noindent Recall that Lebesgue-a.e.\ $\alpha$ satisfies $\sup_{k \in \mathbb{N}} a_{2k}=\infty$ and $\sup_{k \in \mathbb{N}} a_{2k+1}=\infty$. In particular, the extension of $h_p$ is a.e.\ continuous. We conjecture that the condition $\sup_{k \in \mathbb{N}} a_{2k+\varepsilon_p}=\infty$ can be removed, so that Theorem \ref{continuitytheorem} holds for all (including badly approximable) irrationals.

In contrast, $h_p$ has a different behavior at rational numbers. The left-hand limit for $p>0$, and the right-hand limit for $p<0$ exist and are finite at all rationals, and their values are explicitly computable.
\begin{thm}\label{onesidedlimittheorem} Let $a/q \in (0,1)$ and $a'/q'=T^2 (a/q) \in [0,1)$ be reduced rationals, and set
\[ W_p (a/q) = \frac{1}{p} \log \frac{\sum_{N=0}^{q-1} e^{p(S_N(a/q)-\mathrm{sgn}(p)N/(2q))}}{\sum_{N=0}^{q'-1} e^{p(S_N(a'/q')-\mathrm{sgn}(p)N/(2q'))}} + \frac{\lfloor q/a \rfloor (\mathrm{sgn}(p) 2a-1)}{8qq'}, \qquad p \neq \pm \infty,0, \]
and
\[ \begin{split} W_{\infty}(a/q) &= \max_{0 \le N<q} \left( S_N(a/q)- \frac{N}{2q} \right) - \max_{0 \le N<q'} \left( S_N(a'/q')- \frac{N}{2q'} \right) + \frac{\lfloor q/a \rfloor (2a-1)}{8qq'}, \\ W_{-\infty}(a/q) &= \min_{0 \le N<q} \left( S_N(a/q)+ \frac{N}{2q} \right) - \min_{0 \le N<q'} \left( S_N(a'/q')+ \frac{N}{2q'} \right) + \frac{\lfloor q/a \rfloor (-2a-1)}{8qq'}. \end{split} \]
\begin{enumerate}
\item[(i)] If $-\infty \le p<0$, then $\displaystyle{\lim_{r \to (a/q)^+} h_p(r)= W_p(a/q)}$.
\item[(ii)] If $0<p\le \infty$ and $a \neq 1$, then $\displaystyle{\lim_{r \to (a/q)^-} h_p(r)= W_p(a/q)}$.
\end{enumerate}
\end{thm}
\noindent Note that we excluded the rationals $1/q$ for $p>0$. Since $Tr \to \infty$ as $r \to (1/q)^-$, Theorem \ref{asymptoticstheorem} implies that in this case $\lim_{r \to (1/q)^-} h_p(r)=\infty$. As for approaching a rational point from the opposite side, numerical experiments suggest that $h_p$ is right-continuous for $0<p \le \infty$, and left-continuous for $-\infty \le p<0$ at all rationals not of the form $1/q$.

\begin{figure}[ht]
\begin{center}
\begin{subfigure}{0.47\textwidth}
\includegraphics[width=1 \linewidth]{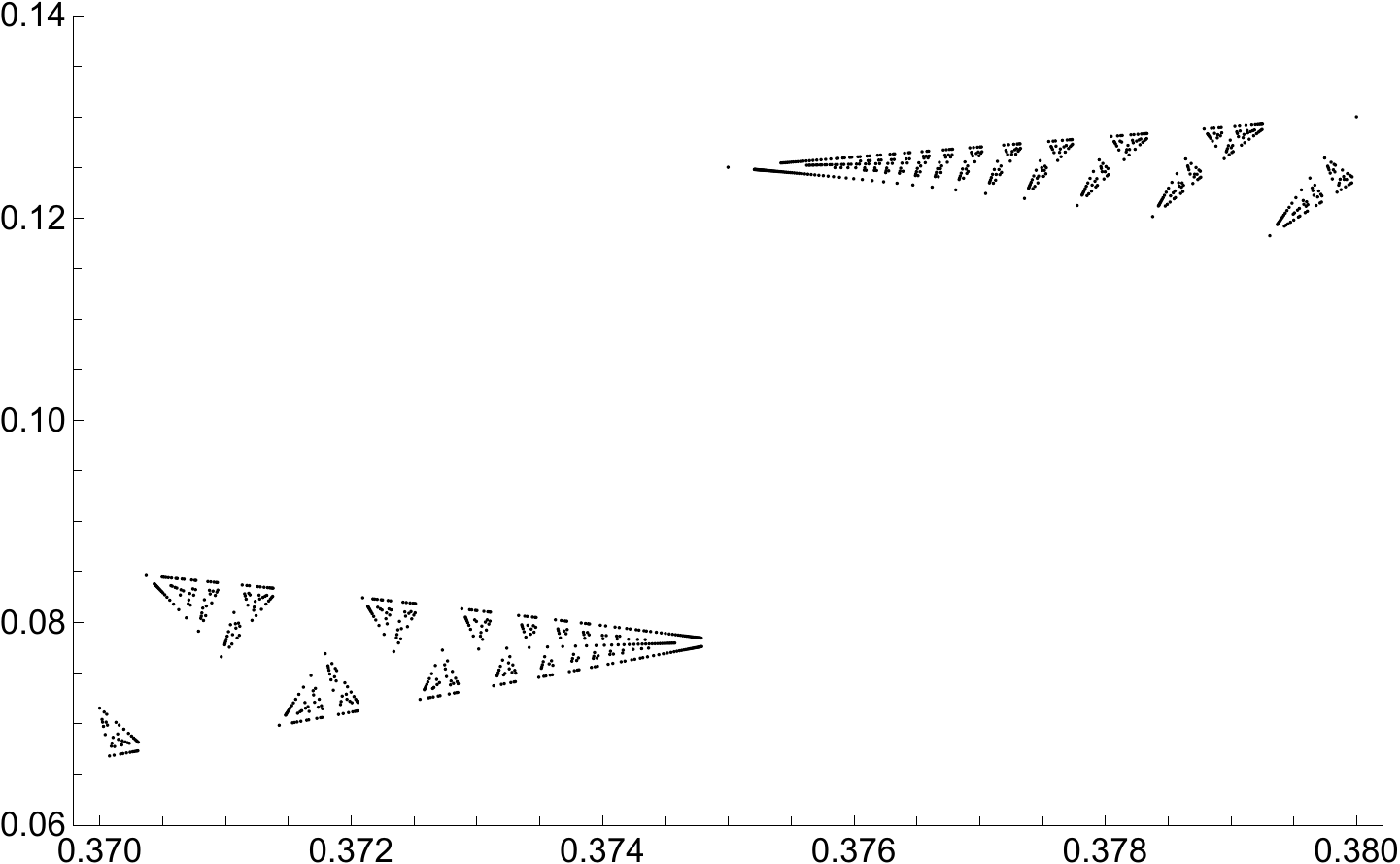}
\caption{$h_{\infty}(r)$}
\end{subfigure}
\hspace{5mm}
\begin{subfigure}{0.47\textwidth}
\includegraphics[width=1 \linewidth]{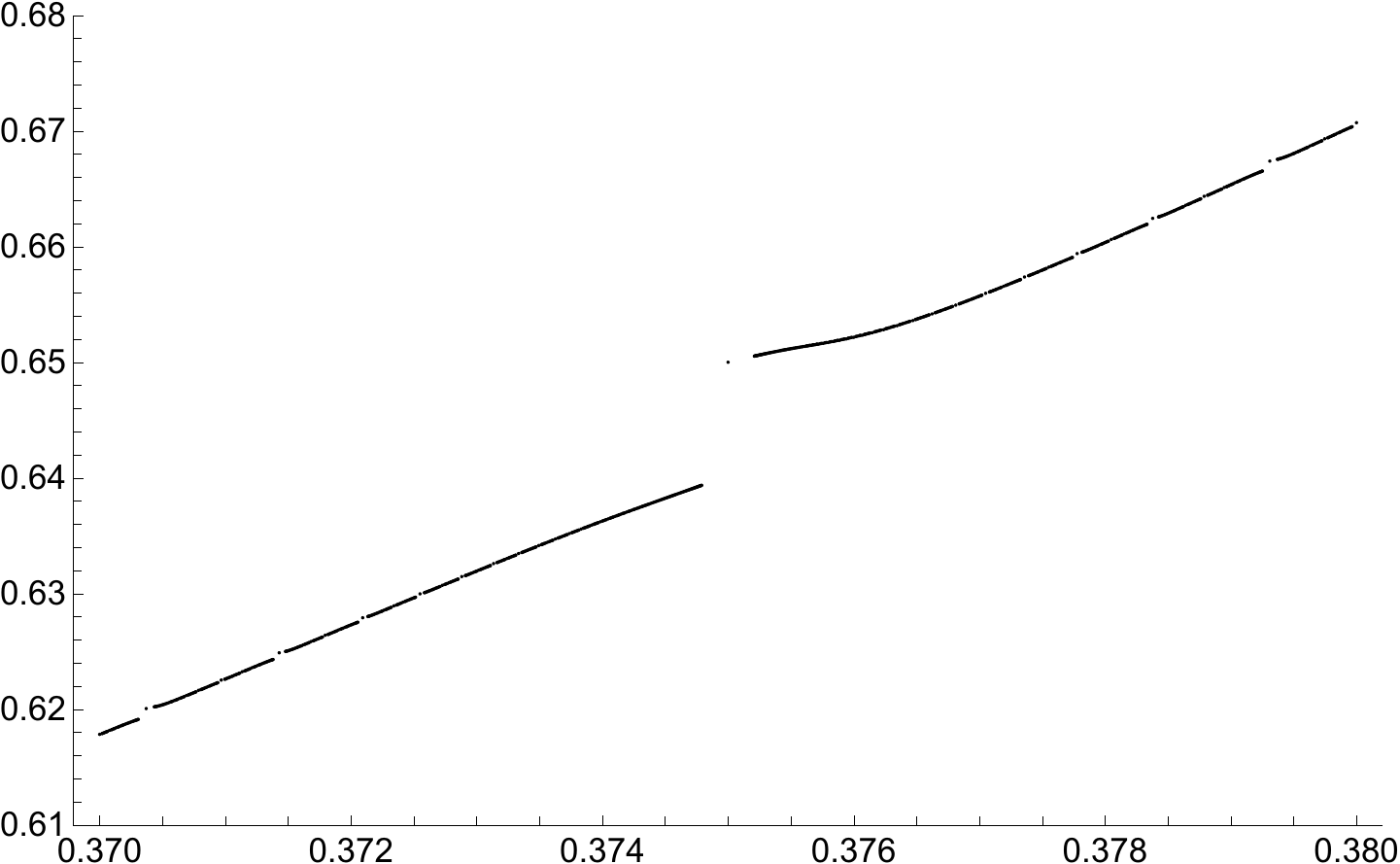}
\caption{$h_2(r)$}
\end{subfigure}
\end{center}
\caption{The functions $h_{\infty}(r)$ and $h_2(r)$ evaluated at all reduced rationals in the interval $[0.37,0.38]$ with denominator at most $600$. At the point $3/8=0.375$ the values are $h_{\infty}(3/8)=1/8$ and $h_2(3/8)=0.650008\ldots$. By Theorem \ref{onesidedlimittheorem}, the left-hand limits at $3/8$ are $W_{\infty}(3/8)=5/64=0.078125$ and $W_2(3/8)=0.640180\ldots$. The graphs suggest right-continuity at $3/8$.}
\end{figure}

In addition to the pathological limit behavior (continuity at irrationals but jumps at rationals), the functions $h_p$ also seem to have a clear self-similar structure, which becomes visible after subtracting the asymptotics established in Theorem \ref{asymptoticstheorem}. A self-similar structure of $\tilde{h}_p$ was numerically observed in \cite{AB1,BD2}. It would be very interesting to actually prove self-similarity, and to gain a deeper understanding of the functions $h_p$ and $\tilde{h}_p$.

\begin{figure}[ht]
\begin{center}
\begin{subfigure}{0.47\textwidth}
\includegraphics[width=1 \linewidth]{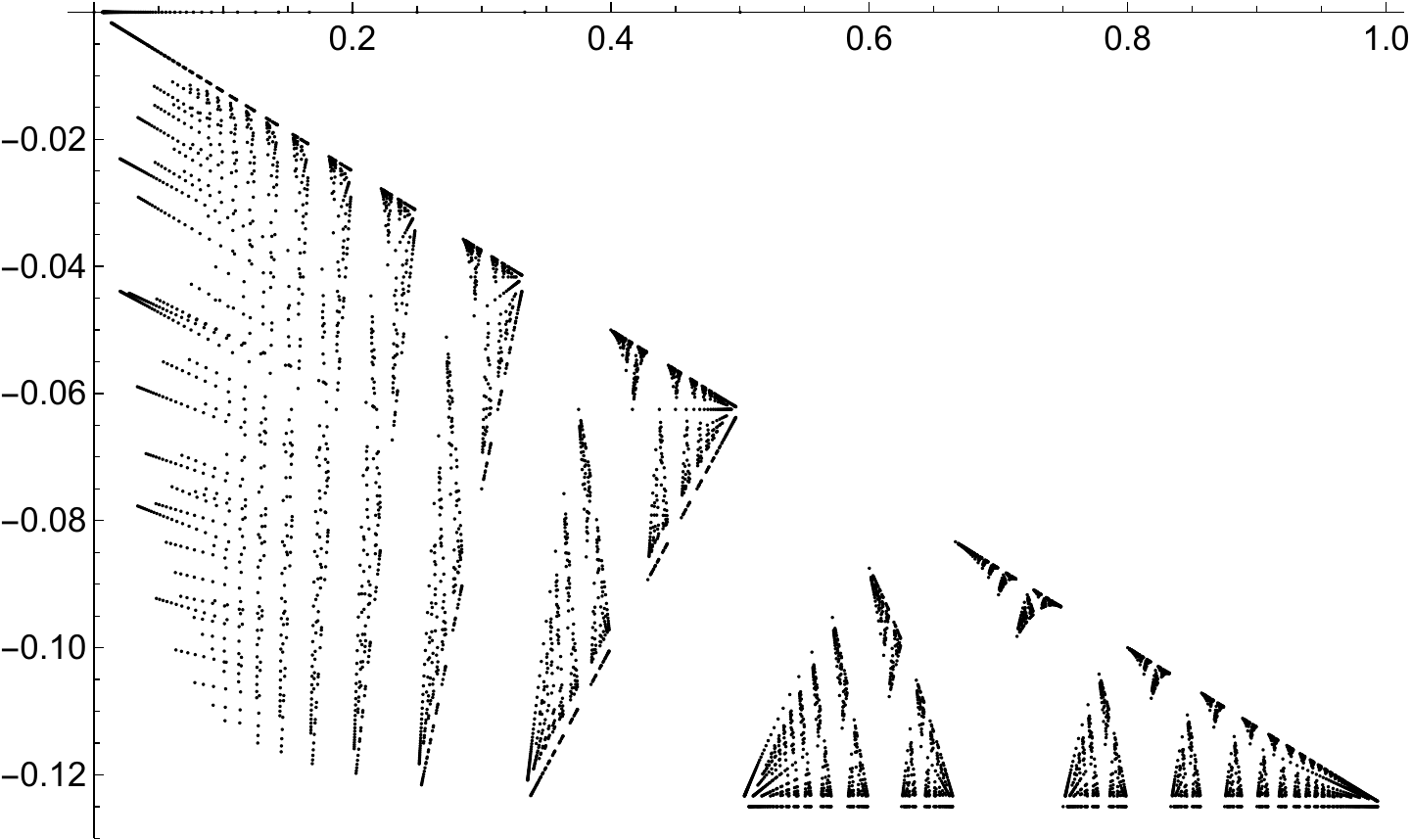}
\caption{$h_{\infty}(r)-\mathds{1}_{\{ Tr \neq 0 \}} \frac{1}{8Tr}$}
\end{subfigure}
\hspace{5mm}
\begin{subfigure}{0.47\textwidth}
\includegraphics[width=1 \linewidth]{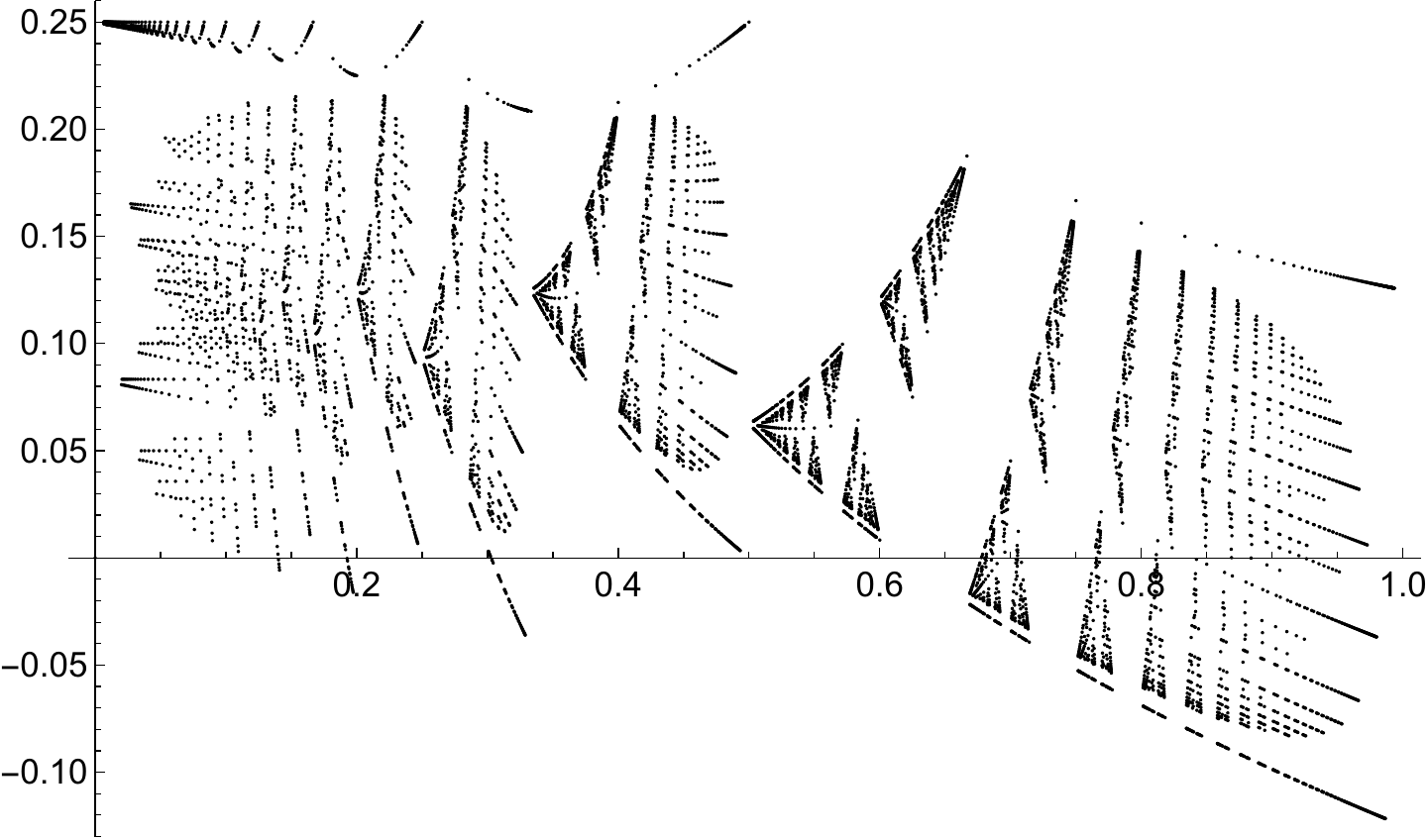}
\caption{$h_{-\infty}(r)+\frac{1}{8r}$}
\end{subfigure}
\end{center}
\begin{center}
\begin{subfigure}{0.47\textwidth}
\includegraphics[width=1 \linewidth]{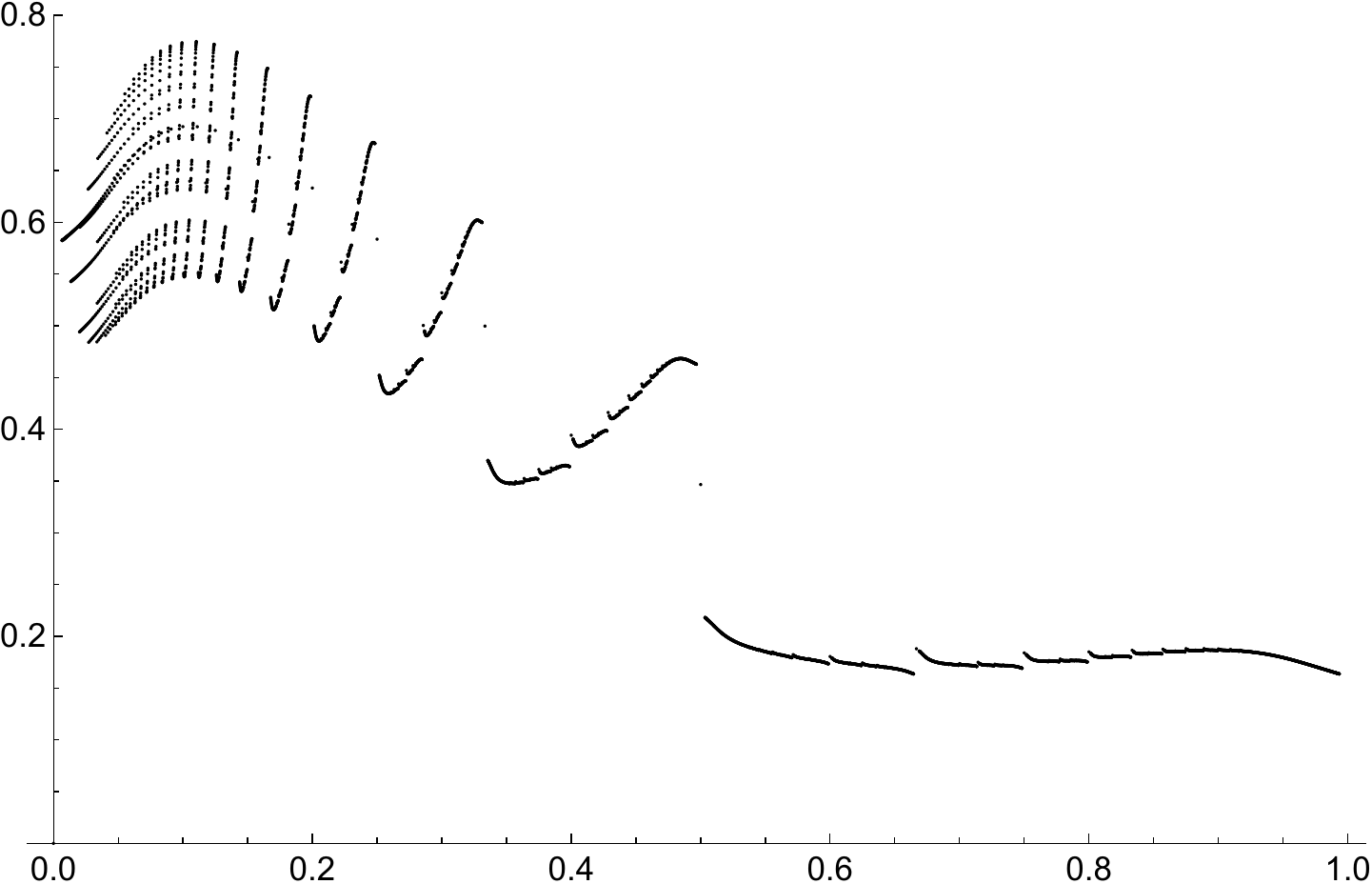}
\caption{$h_2(r)-\mathds{1}_{\{ Tr \neq 0 \}} \left( \frac{1}{8Tr}+\frac{1}{4} \log \frac{1}{Tr}\right)$}
\end{subfigure}
\hspace{5mm}
\begin{subfigure}{0.47\textwidth}
\includegraphics[width=1 \linewidth]{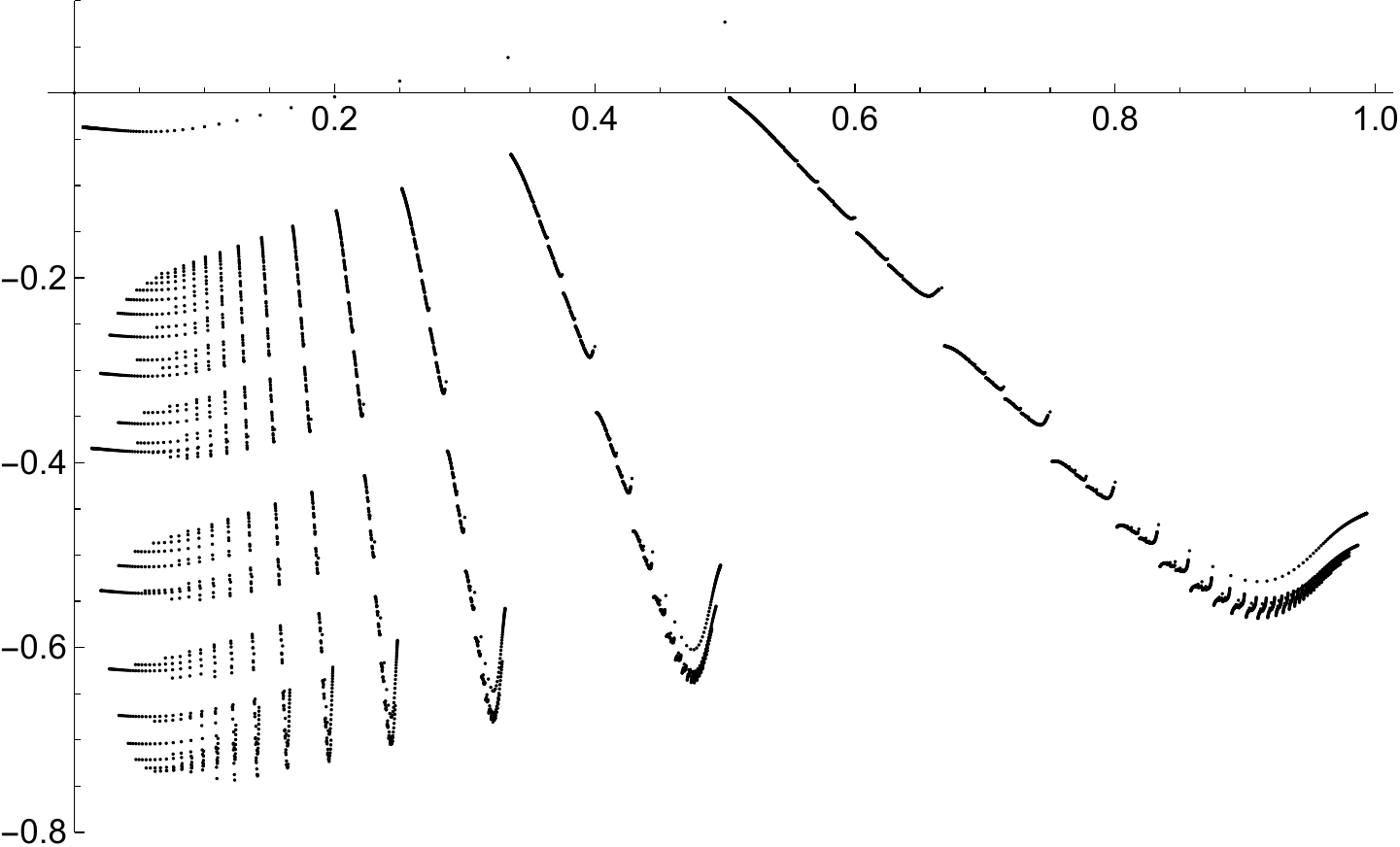}
\caption{$h_{-2}(r)+\frac{1}{8r}+\frac{1}{4}\log \frac{1}{r}$}
\end{subfigure}
\end{center}
\caption{Subtracting the asymptotics from $h_p(r)$ reveals an interesting self-similar structure. Finite $p$ values yield very similar graphs, but the cases $p=\pm \infty$ look markedly different. The four depicted functions are evaluated at all reduced rationals in $[0,1)$ with denominator at most $150$.}
\end{figure}

Given $\alpha \in \mathbb{R}$ and $M \in \mathbb{N}$, as a generalization of $J_p$ we define
\[ J_{p,M}(\alpha) = \left( \sum_{N=0}^{M-1} e^{p S_N (\alpha)} \right)^{1/p}, \qquad p \neq \pm \infty,0, \]
and
\[ J_{\infty,M} (\alpha) = \max_{0 \le N<M} e^{S_N(\alpha)}, \qquad J_{-\infty,M} (\alpha) = \min_{0 \le N<M} e^{S_N(\alpha)} . \]
Let $\tilde{J}_{p,M}(\alpha)$ be defined the same way, with $\tilde{S}_N(\alpha)$ instead of $S_N (\alpha)$. Letting $p_k/q_k$ denote the convergents to $\alpha$, roughly speaking, for $M \approx q_k$ we have $J_{p,M}(\alpha) \approx J_p(p_k/q_k)$ and $\tilde{J}_{p,M}(\alpha) \approx \tilde{J}_p(p_k/q_k)$.

The asymptotics of $\tilde{J}_{p,M}(\alpha)$ as $M \to \infty$ at various irrational $\alpha$ was studied in detail in \cite{AB2,BD2}. In particular, for a quadratic irrational $\alpha$ it was shown that
\[ \log \tilde{J}_{p,M}(\alpha) = \tilde{C}_p (\alpha) \log M +O(\max \{ 1, 1/|p|\}) \]
with some constant $\tilde{C}_p (\alpha)$ and an implied constant depending only on $\alpha$. Moreover, the constants satisfy the relation $\tilde{C}_p(\alpha) + \tilde{C}_{-p}(\alpha)=1$. In this paper, we establish a similar result for $J_{p,M}(\alpha)$.
\begin{thm}\label{quadratictheorem} For any $-\infty \le p \le \infty$, $p \neq 0$ and any quadratic irrational $\alpha$,
\[ \log J_{p,M} (\alpha) = C_p (\alpha) \log M+O(\max \{ 1,1/|p| \}) \]
with some constant $C_p (\alpha)$ and an implied constant depending only on $\alpha$.
\end{thm}
\noindent Relation \eqref{quadraticmaxmin} is a special case of Theorem \ref{quadratictheorem} with $p=\pm \infty$. Note that $0<C_{\infty}(\alpha) \le C_p (\alpha)$ if $p>0$, and $C_p(\alpha) \le C_{-\infty}(\alpha)<0$ if $p<0$. Unlike $C_{\pm \infty}(\alpha)$, we do not know how to compute $C_p(\alpha)$ for finite $p$, even for simple irrationals such as the golden mean.

The constants $C_p(\alpha)$ and $\tilde{C}_p(\alpha)$ are closely related to the limit of the functions $h_p$ and $\tilde{h}_p$ at quadratic irrationals. As an illustration, consider $\sqrt{3}-1=[0;\overline{1,2}]$, and let $p_k/q_k$ denote its convergents. Then $T^2 (p_k/q_k)=p_{k-2}/q_{k-2}$, hence by the definition of $h_p$ and the fact that $\log q_k \sim (k/2) \log (2+\sqrt{3})$,
\[ \sum_{0 \le j <k/2} h_p (p_{k-2j}/q_{k-2j}) = \log J_p (p_k/q_k) = \log J_{p,q_k} (\sqrt{3}) +O(1) = \frac{C_p(\sqrt{3}) \log (2+\sqrt{3})}{2} k +O(1) . \]
Thus if $\lim_{r \to \sqrt{3}-1} h_p(r)$ exists, then its value must be $C_p(\sqrt{3})\log (2+\sqrt{3})$. In particular, while we cannot establish the continuous extension of $h_{\pm \infty}$ to $\sqrt{3}-1$, we know that in case they can be continuously extended to that point, their values must be $h_{\infty}(\sqrt{3}-1)=1/4$ and $h_{-\infty} (\sqrt{3}-1)=-1/12$; this is in good accordance with the numerics. For a general quadratic irrational $\alpha$, the constant $C_p(\alpha)$ can be similarly expressed in terms of the limit of $h_p$ at the points of the finite orbit of $\alpha$ under $T^2$, provided that these limits exist.

\section{Limit laws}\label{limilawsection}

Confirming a conjecture of Bettin and Drappeau \cite{BD2}, Aistleitner and the author \cite{AB1} proved the following limit law for the value distribution of $\tilde{J}_p(r)$ with a random rational $r$; more precisely, for a randomly chosen element of $F_Q = \{ a/q \in [0,1] \, : \mathrm{gcd} (a,q)=1, \,\, 1 \le q \le Q \}$, the set of Farey fractions of order $Q$. If $a/q \sim \mathrm{Unif}(F_Q)$, then for any $0<p \le \infty$,
\[ \frac{\log \tilde{J}_p (a/q) - \tilde{E}_{p,q}}{\tilde{\sigma}_q} \overset{d}{\to} \mathrm{Stab} (1,1) \qquad \textrm{as } Q \to \infty , \]
where $\tilde{E}_{p,q} = \frac{3 \mathrm{Vol}(4_1)}{\pi^3} \log q \log \log q+\tilde{D}_p \log q$ and $\tilde{\sigma}_q=\frac{3\mathrm{Vol}(4_1)}{2 \pi^2} \log q$, with the constant
\begin{equation}\label{Dtilde}
\tilde{D}_p = \frac{3 \mathrm{Vol}(4_1)}{\pi^3} \left( \log \frac{6}{\pi} - \gamma \right) + \frac{12}{\pi^2} \int_0^1 \frac{\tilde{h}_p (x) - \frac{\mathrm{Vol}(4_1)}{4 \pi} \lfloor 1/x \rfloor}{1+x} \, \mathrm{d}x.
\end{equation}
Here $\gamma$ denotes the Euler--Mascheroni constant. This was proved in \cite{AB1} for $p=2$, but the proof works mutatis mutandis for all $0<p \le \infty$. The identity $\tilde{J}_{-p}(r)=q/\tilde{J}_p(r)$ mentioned in Section \ref{quantumsection} means that $\log \tilde{J}_p (a/q) + \log \tilde{J}_{-p} (a/q) = \log q$, and a limit law follows for $-\infty \le p<0$ as well.

In this paper, we show a similar limit law for $J_p(r)$ with a random rational $r$.
\begin{thm}\label{Jptheorem} Let $a/q \sim \mathrm{Unif}(F_Q)$. For any $0<p \le \infty$ and $-\infty \le p'<0$,
\[ \left( \frac{\log J_p (a/q) - E_{p,q}}{\sigma_q}, \frac{\log J_{p'} (a/q) - E_{p',q}}{\sigma_q} \right) \overset{d}{\to} \mathrm{Stab}(1,1) \otimes \mathrm{Stab}(1,-1) \qquad \textrm{as } Q \to \infty , \]
where, for any $p \neq 0$, $E_{p,q} = \mathrm{sgn}(p) \frac{3}{4 \pi^2} \log q \log \log q + D_p \log q$ and $\sigma_q = \frac{3}{8 \pi} \log q$, with the constant
\begin{equation}\label{Dp}
D_p = - \mathrm{sgn} (p) \frac{3}{4 \pi^2}  \left( \gamma+ \log \frac{\pi}{3} \right) + \left\{ \begin{array}{ll} \frac{6}{\pi^2} \int_0^1 \frac{h_p (x) - \frac{1}{8} \lfloor 1/Tx \rfloor}{1+x} \, \mathrm{d}x & \textrm{if } p>0, \\ \frac{6}{\pi^2} \int_0^1 \frac{h_p (x) + \frac{1}{8} \lfloor 1/x \rfloor}{1+x} \, \mathrm{d}x & \textrm{if } p<0 . \end{array} \right.
\end{equation}
\end{thm}
\noindent In particular,
\[ \frac{\log J_p (a/q) - E_{p,q}}{\sigma_q} \overset{d}{\to} \mathrm{Stab} (1,\mathrm{sgn}(p)) \qquad \textrm{as } Q \to \infty . \]
\begin{remark} The identity $J_{-p}(r)=1/J_p(1-r)$ and the fact that $a/q \mapsto 1-a/q$ is a bijection of $F_Q$ show that $\log J_{-p} (a/q)$ and $-\log J_p (a/q)$ are identically distributed. The previous limit law thus implies that $E_{-p,q}=-E_{p,q}$, and consequently $D_{-p}=-D_p$, a relation which is not immediate from the definition \eqref{Dp} of $D_p$.
\end{remark}

The main idea is to consider the telescoping sum $\log \tilde{J}_p(r) = \sum_{j \ge 0} \tilde{h}_p (T^j r)$; note that $\tilde{h}_p (0)=0$. Using the asymptotics \eqref{tildehpasymptotics} and the solution to Zagier's continuity conjecture, for $0<p \le \infty$ we can write $\tilde{h}_p (r) = \frac{\mathrm{Vol}(4_1)}{4 \pi} a_1 + \tilde{g}_p(r)$ with an a.e.\ continuous Lebesgue integrable function $\tilde{g}_p(x)$. Letting $a/q=[0;a_1,a_2,\ldots, a_L]$ be a random fraction, we thus have
\[ \log \tilde{J}_p (a/q) = \frac{\mathrm{Vol}(4_1)}{4 \pi} \sum_{j \ge 0} a_{j+1} + \sum_{j \ge 0} \tilde{g}_p (T^j (a/q)) . \]
The first sum, with suitable centering and scaling, converges in distribution to $\mathrm{Stab}(1,1)$, whereas the second sum, scaled by $\log q$, converges in distribution to a constant. This leads to the limit law for $\log \tilde{J}_p$.

We follow a similar strategy for $J_p$. We consider the telescoping sum $\log J_p(r) = \sum_{j \ge 0} h_p(T^{2j} r)$; note that $h_p(0)=0$. Using Theorems \ref{asymptoticstheorem} and \ref{continuitytheorem}, we can write $h_p(r) = \mathrm{sgn} (p) a_{\varepsilon_p}/8 + g_p(r)$ with an a.e.\ continuous Lebesgue integrable function $g_p(x)$. Letting $a/q=[0;a_1,a_2,\ldots, a_L]$ be a random fraction, we thus have
\[ \log J_p (a/q) = \left\{ \begin{array}{ll} \frac{1}{8} \sum_{j \ge 0} a_{2j+2} + \sum_{j \ge 0} g_p (T^{2j}(a/q)) & \textrm{if } p>0, \\ - \frac{1}{8} \sum_{j \ge 0} a_{2j+1} + \sum_{j \ge 0} g_p (T^{2j} (a/q)) & \textrm{if } p<0 . \end{array} \right. \]
The main difference is that the main term in $\log J_p(a/q)$ now depends only on the partial quotients with even resp.\ odd indices if $p>0$ resp.\ $p<0$. This explains the convergence of the joint distribution to a product measure in Theorem \ref{Jptheorem}.

Classical mixing properties of the sequence of partial quotients lead to similar limit laws for random real numbers.
\begin{thm}\label{JpMtildetheorem} Let $\alpha \sim \mu$ with a Borel probability measure $\mu$ on $[0,1]$ which is absolutely continuous with respect to the Lebesgue measure. For any $0<p \le \infty$,
\[ \frac{\log \tilde{J}_{p,M}(\alpha) - \tilde{E}_{p,M}}{\tilde{\sigma}_M} \overset{d}{\to} \mathrm{Stab}(1,1) \qquad \textrm{as } M \to \infty, \]
where $\tilde{E}_{p,M}=\frac{3 \mathrm{Vol}(4_1)}{\pi^3} \log M \log \log M + \tilde{D}_p \log M$ and $\tilde{\sigma}_M= \frac{3\mathrm{Vol}(4_1)}{2 \pi^2} \log M$, with the constant $\tilde{D}_p$ defined in \eqref{Dtilde}.
\end{thm}
\noindent Formula \eqref{maxtildelimitlaw} is a special case of Theorem \ref{JpMtildetheorem} with $p=\infty$. Since
\[ \log \tilde{J}_{p,M} (\alpha) + \log \tilde{J}_{-p,M} (\alpha) = \log M + o(\log M) \qquad \textrm{in $\mu$-measure}, \]
a similar limit law holds for $\log \tilde{J}_{p,M}(\alpha)$ with $-\infty \le p <0$.
\begin{thm}\label{JpMtheorem} Let $\alpha \sim \mu$ with a Borel probability measure $\mu$ on $[0,1]$ which is absolutely continuous with respect to the Lebesgue measure. For any $0<p \le \infty$ and $-\infty \le p' <0$,
\[ \left( \frac{\log J_{p,M} (\alpha) -E_{p,M}}{\sigma_M}, \frac{\log J_{p',M} (\alpha) - E_{p',M}}{\sigma_M} \right) \overset{d}{\to} \mathrm{Stab}(1,1) \otimes \mathrm{Stab}(1,-1) \qquad \textrm{as } M \to \infty , \]
where, for any $p \neq 0$, $E_{p,M} = \mathrm{sgn}(p) \frac{3}{4 \pi^2} \log M \log \log M + D_p \log M$ and $\sigma_M = \frac{3}{8 \pi} \log M$, with the constant $D_p$ defined in \eqref{Dp}.
\end{thm}
\noindent Theorem \ref{maxmintheorem} is a special case of Theorem \ref{JpMtheorem} with $p=\infty$ and $p'=-\infty$.

\section{The function $h_p$}\label{proofhpsection}

Throughout this section, we fix a real number $\alpha$ and a parameter $-\infty \le p \le \infty$, $p \neq 0$, and define $\varepsilon_p$ as in \eqref{varepsilonp}. If $\alpha \in \mathbb{Q}$, we write its continued fraction expansion in the form $\alpha =[a_0;a_1,a_2,\ldots, a_L]$, and we let $q$ be the denominator of $\alpha$ in its reduced form. If $\alpha \not\in \mathbb{Q}$, we write its continued fraction expansion in the form $\alpha =[a_0;a_1,a_2,\ldots]$, and set $L=\infty$ and $q=\infty$.

The convergents to $\alpha$ are denoted by $p_{\ell}/q_{\ell}=[a_0;a_1,a_2,\ldots, a_{\ell}]$, $0 \le \ell < L+1$. Any integer $0 \le N <q$ can be uniquely written in the form $N=\sum_{\ell=0}^{L-1} b_{\ell}(N) q_{\ell}$, where $0 \le b_0(N) < a_1$ and $0 \le b_{\ell}(N) \le a_{\ell +1}$, $1 \le \ell <L$ are integers which further satisfy the rule that $b_{\ell+1} (N)=a_{\ell+2}$ implies $b_{\ell}(N)=0$. This is the so-called Ostrowski expansion of $N$ with respect to $\alpha$, a special number system tailored to the circle rotation by $\alpha$; in fact, it was first introduced in connection to $S_N(\alpha)$ \cite{OS}. The Ostrowski expansion of course has finitely many terms; more precisely, if $0 \le N <q_K$ with some integer $0 \le K \le L$, then $N=\sum_{\ell=0}^{K-1} b_{\ell}(N) q_{\ell}$.

The distance from the nearest integer function is denoted by $\| \cdot \|$. We will often use the fact that
\[ \frac{1}{a_{\ell +1}+2} \le q_{\ell} \| q_{\ell} \alpha \| \le \frac{1}{a_{\ell +1}}, \qquad 0 \le \ell <L, \]
except if $\ell=0$ and $a_1=1$; however, in the latter case $b_0(N)=0$ for all $0\le N < q$, and $\| q_0 \alpha \|$ does not enter our formulas. Recall also the recursion $q_{\ell+1}=a_{\ell+1}q_{\ell}+q_{\ell-1}$ with initial conditions $q_0=1$, $q_1=a_1$.

One of our main tools is an explicit formula for $S_N(\alpha)$ due to Ostrowski \cite{OS} (see \cite[p.\ 23]{BE} for a more recent proof).
\begin{lem}[Ostrowski]\label{ostrowskilemma} Let $0 \le N <q$ be an integer with Ostrowski expansion $N=\sum_{\ell=0}^{L-1} b_{\ell}(N) q_{\ell}$. Then
\[ S_N (\alpha) = \sum_{\ell=0}^{L-1} (-1)^{\ell+1} b_{\ell}(N) \left( \frac{1-b_{\ell}(N) q_{\ell} \| q_{\ell} \alpha \|}{2} - \| q_{\ell} \alpha \| \sum_{j=0}^{\ell-1} b_j(N) q_j - \frac{\| q_{\ell} \alpha \|}{2} \right) . \]
\end{lem}

\begin{remark}
The alternating factor $(-1)^{\ell+1}$ in Ostrowski's explicit formula is related to the fact that $S_N(\alpha)$ is an odd function in the variable $\alpha$. An application of the second iterate of the Gauss map corresponds to shifting the partial quotients by two indices, leaving the factor $(-1)^{\ell+1}$ unchanged.
\end{remark}

\subsection{Local optimum}

In this section, we ``locally optimize'' $S_N(\alpha)$ by choosing a single Ostrowski digit $b_k(N)$. Note that the $\ell=k$ term in Ostrowski's explicit formula in Lemma \ref{ostrowskilemma} is
\[ \frac{(-1)^{k+1}}{2} \cdot \frac{b_k(N)}{a_{k+1}} \left( 1- \frac{b_k(N)}{a_{k+1}} \right) +O(1) . \]
Given an odd resp.\ even index $k$, we can thus expect a particularly large resp.\ small value of $S_N(\alpha)$ when choosing $b_k(N) \approx a_{k+1}/2$. Lemma \ref{localoptimumlemma} below quantifies how the value of $S_N(\alpha)$ changes as we deviate from the optimal value $a_{k+1}/2$. In particular, in Lemma \ref{localcorollary} below we show that in the sum $\sum_{N=0}^{q-1} e^{p S_N(\alpha)}$ with $p>0$ resp.\ $p<0$, the main contribution comes from the terms with $b_k(N) \approx a_{k+1}/2$.

In the following lemma and in the sequel, we use the natural convention that $b_L(N)<a_{L+1}$ automatically holds.
\begin{lem}\label{localoptimumlemma} Let $0 \le N<q$ be an integer with Ostrowski expansion $N=\sum_{\ell=0}^{L-1} b_{\ell}(N) q_{\ell}$, and let $0 \le k <L$. Define $b_k^*=\lfloor a_{k+1}/2 \rfloor$, and
\[ N^* = \left\{ \begin{array}{ll} N+(b_k^*-b_k(N))q_k & \textrm{if } b_{k+1}(N)<a_{k+2}, \\ N+b_k^* q_k - q_{k+1} & \textrm{if } b_{k+1}(N)=a_{k+2}. \end{array} \right. \]
Then
\[ S_{N^*}(\alpha) - S_N(\alpha) = (-1)^{k+1} \frac{(b_k^*-b_k(N))^2}{2 a_{k+1}} + O \left( \frac{|b_k^*-b_k(N)|}{a_{k+1}} \right) \]
with a universal implied constant.
\end{lem}

\begin{proof} Assume first, that $b_{k+1}(N)<a_{k+2}$. Then $N^*$ is obtained from $N$ by changing the Ostrowski digit $b_k(N)$ to $b_k^*$, and leaving all other Ostrowski digits intact. Applying Ostrowski's explicit formula in Lemma \ref{ostrowskilemma} to $N$ and $N^*$, we deduce
\begin{equation}\label{SN*-SNfirst}
\begin{split} S_{N^*}(\alpha ) - S_N (\alpha) = &(-1)^{k+1} \left( b_k^* \frac{1-b_k^* q_k \| q_k \alpha \|}{2} - b_k(N) \frac{1-b_k(N)q_k \| q_k \alpha \|}{2} \right) \\ &+(-1)^{k} (b_k^*-b_k(N)) \left( \| q_k \alpha \| \sum_{j=0}^{k-1} b_j(N) q_j + \frac{\| q_k \alpha \|}{2} \right) \\ &+\sum_{\ell =k+1}^{L-1} (-1)^{\ell} b_{\ell}(N) \| q_{\ell} \alpha \| (b_k^* -b_k(N)) q_k . \end{split}
\end{equation}
By the rules of the Ostrowski expansion, here $0 \le \sum_{j=0}^{k-1} b_j(N)q_j < q_k$. Therefore the second and the third line in \eqref{SN*-SNfirst} are negligible:
\[ |b_k^*-b_k(N)| \left( \| q_k \alpha \| \sum_{j=0}^{k-1} b_j(N) q_j + \frac{\| q_k \alpha \|}{2} \right) \ll \frac{|b_k^*-b_k(N)|}{a_{k+1}} , \]
and
\[ \left| \sum_{\ell =k+1}^{L-1} (-1)^{\ell} b_{\ell}(N) \| q_{\ell} \alpha \| (b_k^* -b_k(N)) q_k \right| \le |b_k^*-b_k(N)| q_k \sum_{\ell=k+1}^{L-1} \frac{1}{q_{\ell}} \ll \frac{|b_k^*-b_k(N)|}{a_{k+1}} . \]
Note that $b_k^* q_k \| q_k \alpha \| = 1/2+O(1/a_{k+1})$ by the definition of $b_k^*$. The polynomial $F(x)=x(1-x)$ satisfies the identity $F(x)-F(y)=(x-y)^2+(x-y)(1-2x)$, hence
\[ F(b_k^* q_k \| q_k \alpha \|) - F(b_k(N) q_k \| q_k \alpha \|) = (b_k^*-b_k(N))^2 q_k^2 \| q_k \alpha \|^2 + O \left( \frac{|b_k^*-b_k(N)| q_k \| q_k \alpha \|}{a_{k+1}} \right) , \]
and consequently in the first line in \eqref{SN*-SNfirst} we have
\[ \begin{split}  b_k^* \frac{1-b_k^* q_k \| q_k \alpha \|}{2} - b_k(N) \frac{1-b_k(N)q_k \| q_k \alpha \|}{2} &= \frac{F(b_k^* q_k \| q_k \alpha \|) - F(b_k(N) q_k \| q_k \alpha \|)}{2q_k \| q_k \alpha \|} \\ &= \frac{(b_k^* -b_k(N))^2 q_k \| q_k \alpha \|}{2} +O \left( \frac{|b_k^*-b_k(N)|}{a_{k+1}} \right) \\ &= \frac{(b_k^*-b_k(N))^2}{2 a_{k+1}} + O \left( \frac{|b_k^*-b_k(N)|}{a_{k+1}} \right) . \end{split} \]
This finishes the proof in the case $b_{k+1}(N)<a_{k+2}$.

Assume next, that $b_{k+1}(N)=a_{k+2}$. By the rules of the Ostrowski expansion, we necessarily have $b_k(N)=0$, thus $N^*$ is obtained from $N$ by decreasing the digit $b_{k+1}(N)=a_{k+2}$ by one, and changing $b_k(N)=0$ to $b_k^*$. We arrive at a legitimate Ostrowski expansion of $N^*$; in particular, $b_{\ell}(N^*)=b_{\ell}(N)$ for all $\ell \neq k,k+1$. Applying Ostrowski's explicit formula in Lemma \ref{ostrowskilemma} to $N$ and $N^*$, we deduce
\[ \begin{split} &S_{N^*}(\alpha) - S_N(\alpha) = \\ &(-1)^{k+1} b_k^* \left( \frac{1-b_k^* q_k \| q_k \alpha \|}{2} -\| q_k \alpha \| \sum_{j=0}^{k-1} b_j(N) q_j  - \frac{\| q_k \alpha \|}{2} \right) \\ &+(-1)^{k+2} \left( (a_{k+2}-1) \frac{1-(a_{k+2}-1) q_{k+1} \| q_{k+1} \alpha \|}{2} - a_{k+2} \frac{1-a_{k+2} q_{k+1} \| q_{k+1} \alpha \|}{2} \right) \\ &+(-1)^{k+2} \left( -(a_{k+2}-1) \| q_{k+1} \alpha \| \left( \sum_{j=0}^{k-1} b_j(N) q_j +b_k^* q_k + \frac{1}{2} \right) +a_{k+2} \| q_{k+1} \alpha \| \left( \sum_{j=0}^{k-1} b_j(N) q_j +\frac{1}{2} \right) \right) \\ &+\sum_{\ell=k+2}^{L-1} (-1)^{\ell} b_{\ell}(N) \| q_{\ell} \alpha \| (b_k^* q_k - q_{k+1}) . \end{split} \]
Straightforward computation shows that the first line in the previous formula is $(-1)^{k+1} a_{k+1}/8+O(1)$, and all other lines are $O(1)$.
\end{proof}

\begin{lem}\label{localcorollary} Let $0 \le k < K \le L$ be integers such that $a_{k+1} \ge A$ with a large universal constant $A>1$. If $p \neq \pm \infty$ and $k+1 \equiv \varepsilon_p \pmod{2}$, then
\[ \sum_{\substack{0 \le N <q_K \\ |b_k(N)-a_{k+1}/2| > \max \{ 10, 10/\sqrt{|p|} \} \sqrt{a_{k+1} \log a_{k+1}}}} e^{p S_N(\alpha )} \le a_{k+1}^{-48 \max \{ |p|, 1 \}} \sum_{0 \le N<q_K} e^{p S_N(\alpha)} . \]
If $k$ is odd, then
\[ \max_{\substack{0 \le N < q_K \\ |b_k(N)-a_{k+1}/2| > 10 \sqrt{a_{k+1} \log a_{k+1}}}} e^{S_N(\alpha)} \le a_{k+1}^{-48} \max_{0 \le N <q_K} e^{S_N(\alpha)}. \]
If $k$ is even, then
\[ \min_{\substack{0 \le N < q_K \\ |b_k(N)-a_{k+1}/2| > 10 \sqrt{a_{k+1} \log a_{k+1}}}} e^{S_N(\alpha)} \ge a_{k+1}^{48} \min_{0 \le N <q_K} e^{S_N(\alpha)} . \]
\end{lem}

\begin{proof} We give a detailed proof in the case $0<p<\infty$. The proof for $-\infty < p<0$ is entirely analogous, whereas the claims on the maximum and the minimum follow from letting $p \to \pm \infty$.

Assume thus that $0<p<\infty$, and that $k$ is odd. Set $Z=\sum_{0 \le N<q_K} e^{p S_N(\alpha)}$, and consider the sets
\[ \begin{split} H_k (b) &= \{ 0 \le N < q_K \, : \, b_k (N)=b \}, \\ H_k^*(0) &= \{ 0 \le N < q_K \, : \, b_k (N)=0, \,\, b_{k+1}(N)<a_{k+2} \} , \\ H_k^{**}(0) &= \{ 0 \le N < q_K \, : \, b_k (N)=0, \,\, b_{k+1}(N)=a_{k+2} \} . \end{split} \]
Let $|b- a_{k+1} /2|>\max \{ 10, 10/\sqrt{p} \} \sqrt{a_{k+1} \log a_{k+1}}$ and $b \neq 0$. Then the map $H_k(b) \to H_k(\lfloor a_{k+1}/2 \rfloor)$, $N \mapsto N+(\lfloor a_{k+1}/2 \rfloor -b)q_k$ is injective, and by Lemma \ref{localoptimumlemma},
\[ \sum_{N \in H_k(b)} e^{pS_N(\alpha )} \le \sum_{N \in H_k(\lfloor a_{k+1}/2 \rfloor)} e^{p(S_N(\alpha) - (b-a_{k+1}/2)^2/(2.001a_{k+1}))} \le a_{k+1}^{-49.9 \max \{ p, 1 \}} Z. \]
The map $H_k^*(0) \to H_k(\lfloor a_{k+1}/2 \rfloor)$, $N \mapsto N+\lfloor a_{k+1}/2 \rfloor q_k$ is injective, and by Lemma \ref{localoptimumlemma},
\[ \sum_{N \in H_k^*(0)} e^{p S_N(\alpha )} \le \sum_{N \in H_k(\lfloor a_{k+1}/2 \rfloor)} e^{p(S_N(\alpha) - (a_{k+1}/2)^2/(2.001a_{k+1}))} \le e^{-pa_{k+1}/8.004} Z. \]
The map $H_k^{**}(0) \to H_k(\lfloor a_{k+1}/2 \rfloor)$, $N \mapsto N + \lfloor a_{k+1}/2 \rfloor q_k - q_{k+1}$ is injective, and by Lemma \ref{localoptimumlemma},
\[ \sum_{N \in H_k^{**}(0)} e^{p S_N(\alpha )} \le \sum_{N \in H_k(\lfloor a_{k+1}/2 \rfloor)} e^{p(S_N(\alpha) - (a_{k+1}/2)^2/(4a_{k+1}))} \le e^{-pa_{k+1}/8.004} Z. \]
Note that $e^{-pa_{k+1}/8.004} \le a_{k+1}^{-49.9 \max \{ p,1 \}}$ provided that $|0-a_{k+1}/2|>\max \{ 10,10/\sqrt{p} \} \sqrt{a_{k+1} \log a_{k+1}}$. As the number of possible values of $b$ is at most $a_{k+1}-1$, the previous three formulas lead to
\[ \sum_{\substack{0 \le N <q_K \\ |b_k(N)- a_{k+1}/2|> \max \{ 10, 10/\sqrt{p} \} \sqrt{a_{k+1} \log a_{k+1}}}} e^{p S_N(\alpha)} \le (a_{k+1}+1) a_{k+1}^{-49.9 \max\{ p,1 \}} Z \le a_{k+1}^{-48 \max \{ p ,1 \}} Z. \]
\end{proof}

\subsection{Factorization of $J_p$}

In this section, we establish a factorization of $\sum_{0 \le N<q_K} e^{p S_N(\alpha)}$ into a product of two sums up to a small error. The main point of Lemma \ref{factorlemma} below is that the first main factor depends only on the first $k$ partial quotients of $\alpha$. In the special case of a rational $\alpha$ and $K=L$, we obtain a factorization of $J_p(\alpha)$.

\begin{lem}\label{factorlemma} Let $0 \le k < K \le L$ be integers such that $a_{k+1} \ge A \max \{ 1, \frac{1}{|p|} \log \frac{1}{|p|} \}$ with a large universal constant $A>1$. If $p \neq \pm \infty$ and $k+1 \equiv \varepsilon_p \pmod{2}$, then
\[ \begin{split} \Bigg( \sum_{0 \le N <q_K} e^{p S_N(\alpha)} \Bigg)^{1/p} = &\left( 1+O \left( \sqrt{\frac{\log a_{k+1}}{\min \{1,|p| \} a_{k+1}}} \right) \right) \Bigg( \sum_{0 \le N <q_k} e^{p (S_N(p_k/q_k) + (-1)^k N/(2q_k))} \Bigg)^{1/p}\\ &\times \Bigg( \sum_{\substack{0 \le N < q_K \\ b_0(N) = \cdots = b_{k-1}(N) =0 \\ |b_k(N)-a_{k+1}/2| \le \max \{ 10,10/\sqrt{|p|} \} \sqrt{a_{k+1} \log a_{k+1}}}} e^{p S_N(\alpha)} \Bigg)^{1/p} . \end{split} \]
If $k$ is odd, then
\[ \begin{split} \max_{0 \le N <q_K} e^{S_N (\alpha)} =  &\left( 1+O \left( \sqrt{\frac{\log a_{k+1}}{a_{k+1}}} \right) \right) \max_{0 \le N<q_k} e^{S_N(p_k/q_k) - N/(2q_k)} \\ &\times \max_{\substack{0 \le N <q_K \\ b_0(N)=\cdots =b_{k-1}(N)=0 \\ |b_k(N)-a_{k+1}/2| \le \max \{ 10,10/\sqrt{|p|} \} \sqrt{a_{k+1} \log a_{k+1}}}} e^{S_N(\alpha)} . \end{split} \]
If $k$ is even, then
\[ \begin{split} \min_{0 \le N <q_K} e^{S_N (\alpha)} = & \left( 1+O \left( \sqrt{\frac{\log a_{k+1}}{a_{k+1}}} \right) \right) \min_{0 \le N<q_k} e^{S_N(p_k/q_k) + N/(2q_k)} \\ &\times \min_{\substack{0 \le N <q_K \\ b_0(N)=\cdots =b_{k-1}(N)=0 \\ |b_k(N)-a_{k+1}/2| \le \max \{ 10,10/\sqrt{|p|} \} \sqrt{a_{k+1} \log a_{k+1}}}} e^{S_N(\alpha)} . \end{split} \]
All implied constants are universal.
\end{lem}

\noindent We mention that the condition $|b_k(N)-a_{k+1}/2| \le \max \{ 10,10/\sqrt{|p|} \} \sqrt{a_{k+1} \log a_{k+1}}$ in the summations could be removed using a straightforward modification of Lemma \ref{localcorollary}, but we will not need this fact. We give the proof after a preparatory lemma.
\begin{lem}\label{SNlemma} Let $0 \le N <q$ be an integer with Ostrowski expansion $N=\sum_{\ell=0}^{L-1} b_{\ell}(N) q_{\ell}$. Let $0 \le k <L$, and set $N_1=\sum_{\ell=0}^{k-1} b_{\ell}(N) q_{\ell}$ and $N_2=\sum_{\ell =k}^{L-1} b_{\ell}(N) q_{\ell}$. Then
\[ S_N(\alpha) = S_{N_1}(\alpha) + S_{N_2} (\alpha) + (-1)^k b_k(N) \| q_k \alpha \| N_1 + O \left( \frac{1}{a_{k+1}} \right) \]
with a universal implied constant.
\end{lem}

\begin{proof} Apply Ostrowski's explicit formula in Lemma \ref{ostrowskilemma} to $N$, and consider the sum over $0 \le \ell \le k-1$ and $k \le \ell <L$ separately. The sum over $0 \le \ell \le k-1$ is precisely $S_{N_1}(\alpha)$. For $k \le \ell <L$ we have
\[ \sum_{j=0}^{\ell -1} b_j(N) q_j = \sum_{j=0}^{k-1} b_j(N) q_j + \sum_{j=k}^{\ell -1} b_j(N) q_j = N_1 + \sum_{j=0}^{\ell -1} b_j(N_2) q_j , \]
hence
\[ S_N(\alpha) = S_{N_1} (\alpha) + S_{N_2} (\alpha ) + \sum_{\ell=k}^{L-1} (-1)^{\ell} b_{\ell}(N) \| q_{\ell} \alpha \| N_1 . \]
Since $N_1<q_k$, the terms $k+1 \le \ell <L$ in the previous formula satisfy
\[ \left| \sum_{\ell=k+1}^{L-1} (-1)^{\ell}  b_{\ell}(N) \| q_{\ell} \alpha \| N_1 \right| \le \sum_{\ell=k+1}^{L-1} \frac{q_k}{q_{\ell}} \ll \frac{1}{a_{k+1}}, \]
and the claim follows.
\end{proof}

\begin{proof}[Proof of Lemma \ref{factorlemma}] It is enough prove the lemma for finite $p$. The claims on the maximum and the minimum then follow from letting $p \to \pm \infty$.

Lemma \ref{localcorollary} shows that
\begin{equation}\label{corollaryfactorization}
\sum_{0 \le N <q_K} e^{p S_N(\alpha)} = \left( 1+O \left( a_{k+1}^{-48 \max \{ |p|,1 \}} \right) \right) \sum_{\substack{0 \le N <q_K \\ |b_k(N) - a_{k+1}/2| \le \max \{ 10,10/\sqrt{|p|} \} \sqrt{a_{k+1} \log a_{k+1}}}} e^{p S_N(\alpha)} .
\end{equation}

Let $N_1, N_2$ be as in Lemma \ref{SNlemma}. The map $N \mapsto (N_1, N_2)$ is a bijection from
\[ \left\{ 0 \le N <q_K \, : \, |b_k(N) - a_{k+1}/2| \le \max \{ 10, 10/\sqrt{|p|} \} \sqrt{a_{k+1} \log a_{k+1}} \right\} \]
to the product set
\[ [0,q_k) \times \left\{ 0 \le N < q_K : \begin{array}{c} b_0(N)=\cdots=b_{k-1}(N)=0, \\ |b_k(N) - a_{k+1}/2| \le \max \{ 10, 10/\sqrt{|p|} \} \sqrt{a_{k+1} \log a_{k+1}} \end{array} \right\} . \]
For every such $N$,
\[ \begin{split} (-1)^k b_k(N) \| q_k \alpha \| N_1 &= (-1)^k \frac{a_{k+1}}{2} \| q_k \alpha \| N_1 + O \left( \max \{ 1,1/\sqrt{|p|} \} \sqrt{a_{k+1} \log a_{k+1}} \| q_k \alpha \| q_k \right) \\ &= (-1)^k \frac{N_1}{2q_k} + O \left( \max \{ 1,1/\sqrt{|p|} \} \sqrt{\frac{\log a_{k+1}}{a_{k+1}}} \right) . \end{split} \]
Therefore by Lemma \ref{SNlemma},
\[ S_N(\alpha) = S_{N_1}(\alpha) +S_{N_2}(\alpha) + (-1)^k \frac{N_1}{2 q_k} + O \left( \sqrt{\frac{\log a_{k+1}}{\min \{ 1,|p| \} a_{k+1}}} \right) , \]
and consequently
\[ \begin{split} \sum_{\substack{0 \le N<q_K \\ |b_k(N) -a_{k+1}/2| \le \max \{ 10, 10/\sqrt{|p|} \} \sqrt{a_{k+1} \log a_{k+1}}}} e^{p S_N(\alpha)} = & \sum_{0 \le N<q_k} e^{p (S_N(\alpha) + (-1)^k N/(2q_k))} \\ &\times \sum_{\substack{0 \le N <q_K \\ b_0(N)=\cdots=b_{k-1}(N)=0 \\ |b_k(N)-a_{k+1}/2| \le \max \{ 10, 10/\sqrt{|p|} \} \sqrt{a_{k+1} \log a_{k+1}}}} e^{pS_N(\alpha)} \\ &\times \exp \left( O \left( |p| \sqrt{\frac{\log a_{k+1}}{\min \{ 1,|p| \} a_{k+1}}} \right) \right) . \end{split} \]
Substituting this in \eqref{corollaryfactorization} gives
\[ \begin{split} \Bigg( \sum_{0 \le N <q_K} e^{p S_N(\alpha)} \Bigg)^{1/p} = &\left( 1+O \left( \sqrt{\frac{\log a_{k+1}}{\min \{1,|p| \} a_{k+1}}} \right) \right) \Bigg( \sum_{0 \le N <q_k} e^{p (S_N(\alpha) + (-1)^k N/(2q_k))} \Bigg)^{1/p}\\ &\times \Bigg( \sum_{\substack{0 \le N < q_K \\ b_0(N) = \cdots = b_{k-1}(N) =0 \\ |b_k(N)-a_{k+1}/2| \le \max \{ 10,10/\sqrt{|p|} \} \sqrt{a_{k+1} \log a_{k+1}}}} e^{p S_N(\alpha)} \Bigg)^{1/p} . \end{split} \]

It remains to replace $\alpha$ by $p_k/q_k$ in the first main factor in the previous formula. For any $1 \le n<q_k$, we have $|n\alpha -np_k/q_k| = (n/q_k) \| q_k \alpha \| <1/q_k$, and $np_k/q_k$ is not an integer. In particular, there is no integer between $n \alpha$ and $np_k/q_k$, so
\[ \{ n \alpha \} - \left\{ \frac{np_k}{q_k} \right\} = n \alpha - \frac{np_k}{q_k} = \frac{n}{q_k} (-1)^k \| q_k \alpha \|. \]
Therefore for any $0 \le N<q_k$,
\begin{equation}\label{SNalpha-SNpkqk}
S_N(\alpha) - S_N (p_k/q_k) = \sum_{n=1}^N \frac{n}{q_k} (-1)^k \| q_k \alpha \| = O \left( \frac{1}{a_{k+1}} \right) .
\end{equation}
Replacing $\alpha$ by $p_k/q_k$ thus introduces a negligible multiplicative error $1+O(1/a_{k+1})$.
\end{proof}

\subsection{The matching lemma}\label{matchingsection}

Assume now that $\alpha \in (0,1)$, and recall that we write its continued fraction expansion in the form $\alpha = [0;a_1,a_2,\ldots, a_L]$ (if $\alpha \in \mathbb{Q}$) or $\alpha = [0;a_1,a_2,\ldots ]$ (if $\alpha \not\in \mathbb{Q}$), with convergents $p_{\ell}/q_{\ell} = [0;a_1,a_2,\ldots, a_{\ell}]$. Let $\alpha'=T^2 \alpha$, where $T^2$ is the second iterate of the Gauss map $T$. Then $\alpha'=[0;a_3,a_4,\ldots, a_L]$ if $\alpha \in \mathbb{Q}$, with the convention that $\alpha'=0$ if $L \le 2$, and $\alpha'=[0;a_3,a_4,\ldots]$ if $\alpha \not\in \mathbb{Q}$. Let $q'$ denote the denominator of $\alpha'$ in its reduced form if $\alpha \in \mathbb{Q}$, and let $q'=\infty$ if $\alpha \not\in \mathbb{Q}$. Let $p_{\ell}'/q_{\ell}'=[0;a_3,a_4,\ldots, a_{\ell}]$, $3 \le \ell <L+1$ and $p_2'=0$, $q_2'=1$ denote the convergents to $\alpha'$. The Ostrowski expansion of integers $0 \le N <q'$ with respect to $\alpha'$ will be written as $N=\sum_{\ell=2}^{L-1} b_{\ell}'(N) q_{\ell}'$. Note that $0 \le b_2'(N)<a_3$ and $0 \le b_{\ell}'(N) \le a_{\ell +1}$, $3 \le \ell <L$.

Given an integer $0 \le N <q$ with Ostrowski expansion $N=\sum_{\ell=0}^{L-1} b_{\ell}(N) q_{\ell}$ with respect to $\alpha$ such that $b_2(N)<a_3$, define $N'=\sum_{\ell =2}^{L-1} b_{\ell}(N) q_{\ell}'$. Note that this is a legitimate Ostrowski expansion with respect to $\alpha'$, that is, $b_{\ell}'(N')=b_{\ell}(N)$ for all $2 \le \ell <L$. The map $N \mapsto N'$, from $\{ 0 \le N <q \, : \, b_2(N)<a_3 \}$ to $[0,q')$ is surjective but not injective (as it forgets the digits $b_0(N)$ and $b_1(N)$), and provides a natural way to match certain terms of the sum $\sum_{0 \le N <q} e^{pS_N(\alpha)}$ to terms of the sum $\sum_{0 \le N<q'} e^{p S_N (\alpha')}$. By comparing $S_N(\alpha)$ to $S_{N'}(\alpha')$, the following ``matching lemma'' is a key ingredient in the study of the function $h_p$.
\begin{lem}\label{matchinglemma} Let $0 \le N <q$ be an integer with Ostrowski expansion $N=\sum_{\ell=0}^{L-1}b_{\ell}(N) q_{\ell}$ with respect to $\alpha$ such that $b_2(N)<a_3$. Then
\[ S_N(\alpha) - S_{N'}(\alpha') = \sum_{\ell=0}^1 (-1)^{\ell+1} b_{\ell}(N) \left( \frac{1-b_{\ell}(N) q_{\ell} \| q_{\ell} \alpha \|}{2} - \| q_{\ell} \alpha \| \sum_{j=0}^{\ell-1} b_j(N) q_j - \frac{\| q_{\ell} \alpha \|}{2} \right) +O(1) . \]
If in addition $b_0(N)=\cdots =b_{k-1}(N)=0$ with some $k \ge 2$, then
\[ S_N(\alpha) - S_{N'}(\alpha') = a_1 \frac{(-1)^{k+1}p_k b_k(N)/a_{k+1} -(b_k(N)/a_{k+1})^2}{2 q_k q_k'} + O \left( \frac{1}{q_{k+1}'} \right) . \]
The implied constants are universal.
\end{lem}

\begin{proof} Since $p_{\ell}', q_{\ell}'$ satisfy the same second order linear recursion of which $p_{\ell}, q_{\ell}$ are linearly independent solutions, they are linear combinations of $p_{\ell}, q_{\ell}$. Indeed, one readily checks that
\begin{equation}\label{pell'qell'}
p_{\ell}'= (a_1 a_2 +1) p_{\ell} - a_2 q_{\ell} \quad \textrm{and} \quad q_{\ell}'=q_{\ell} -a_1 p_{\ell} \quad \textrm{for all } 2 \le \ell < L+1 .
\end{equation}
Now let $2 \le j \le \ell <L$ be integers. We claim that if either $\ell \ge 3$, or $\ell=2$ and $a_3>1$, then
\begin{equation}\label{qjdeltaell-qj'deltaell'}
\left| q_j \| q_{\ell} \alpha \| - q_j' \| q_{\ell}' \alpha' \| \right| \le \frac{2a_1}{q_{j+1} q_{\ell+1}'} .
\end{equation}
Set $R=[a_{\ell+1};a_{\ell+2}, \ldots, a_L]$ resp.\ $R=[a_{\ell+1}; a_{\ell+2},\ldots]$ if $\alpha \in \mathbb{Q}$ resp.\ $\alpha \not\in \mathbb{Q}$. A classical identity of continued fractions states that $\| q_{\ell} \alpha \| = 1/(R q_{\ell}+q_{\ell-1})$ and $\| q_{\ell}' \alpha' \| = 1/(Rq_{\ell}'+q_{\ell-1}')$. Formula \eqref{pell'qell'} thus leads to
\[ q_j \| q_{\ell} \alpha \| - q_j' \| q_{\ell}' \alpha' \| = a_1 \frac{R q_j q_{\ell} \left( \frac{p_j}{q_j} - \frac{p_{\ell}}{q_{\ell}} \right) + q_j q_{\ell-1} \left( \frac{p_j}{q_j} - \frac{p_{\ell-1}}{q_{\ell-1}} \right)}{(Rq_{\ell}+q_{\ell-1})(Rq_{\ell}'+q_{\ell-1}')} . \]
Observe that $R \ge a_{\ell+1}$, and recall the identity $|q_{\ell} p_{\ell-1} - q_{\ell -1} p_{\ell}|=1$. If $j=\ell$, we thus have
\[ | q_{\ell} \| q_{\ell} \alpha \| - q_{\ell}' \| q_{\ell}' \alpha' \| | = a_1 \frac{1}{(R q_{\ell}+q_{\ell-1})(R q_{\ell}'+q_{\ell-1}')} \le \frac{a_1}{q_{\ell+1} q_{\ell+1}'}, \]
as claimed. If $j=\ell-1$, then
\[ | q_{\ell-1} \| q_{\ell} \alpha \| - q_{\ell-1}' \| q_{\ell}' \alpha' \| | = a_1 \frac{R}{(R q_{\ell}+q_{\ell-1})(R q_{\ell}'+q_{\ell-1}')} \le \frac{a_1}{q_{\ell} q_{\ell+1}'}, \]
as claimed. If $j \le \ell -2$, we can use $|p_j/q_j - p_{\ell}/q_{\ell}| \le 2 |\alpha - p_j/q_j|$ and $|p_j/q_j - p_{\ell-1}/q_{\ell-1}| \le 2 |\alpha - p_j/q_j|$ to deduce
\[ | q_j \| q_{\ell} \alpha \| - q_j' \| q_{\ell}' \alpha' \| | \le a_1 \frac{(R q_j q_{\ell} + q_j q_{\ell-1}) 2 \left| \alpha - \frac{p_j}{q_j} \right|}{(R q_{\ell}+q_{\ell-1})(R q_{\ell}'+q_{\ell-1}')} = \frac{2a_1 \| q_j \alpha \|}{Rq_{\ell}'+q_{\ell-1}'} \le \frac{2a_1}{q_{j+1} q_{\ell+1}'}, \]
as claimed. This finishes the proof of \eqref{qjdeltaell-qj'deltaell'}.

We now prove the lemma. Since $b_{\ell}'(N')=b_{\ell}(N)$ for all $2 \le \ell <L$, Ostrowski's explicit formula in Lemma \ref{ostrowskilemma} gives
\[ S_{N'}(\alpha') = \sum_{\ell=2}^{L-1} (-1)^{\ell+1} b_{\ell}(N) \left( \frac{1-b_{\ell}(N) q_{\ell}' \| q_{\ell}' \alpha' \|}{2} - \| q_{\ell}' \alpha' \| \sum_{j=2}^{L-1} b_j(N) q_j' - \frac{\| q_{\ell}' \alpha' \|}{2} \right) . \]
Consequently,
\[ \begin{split} S_N(\alpha) - S_{N'}(\alpha') = &\sum_{\ell=0}^1 (-1)^{\ell+1} b_{\ell}(N) \left( \frac{1-b_{\ell}(N) q_{\ell} \| q_{\ell} \alpha \|}{2} - \| q_{\ell} \alpha \| \sum_{j=0}^{\ell-1} b_j(N) q_j - \frac{\| q_{\ell} \alpha \|}{2} \right) \\ &+\sum_{\ell=2}^{L-1} (-1)^{\ell+1} b_{\ell}(N) \Bigg( \frac{b_{\ell}(N) (q_{\ell}' \| q_{\ell}' \alpha' \| - q_{\ell} \| q_{\ell} \alpha \|)}{2} -\| q_{\ell} \alpha \| \sum_{j=0}^1 b_j(N) q_j \\ &\hspace{40mm}+ \sum_{j=2}^{\ell-1} b_j(N) \left( q_j' \| q_{\ell}' \alpha' \| - q_j \| q_{\ell} \alpha \| \right) + \frac{\| q_{\ell}' \alpha' \| - \| q_{\ell} \alpha \|}{2} \Bigg) . \end{split} \]
By the estimate \eqref{qjdeltaell-qj'deltaell'} and the fact that $q_{\ell+1} \ge q_2 q_{\ell+1}'$ (which can be seen e.g.\ by induction), the absolute value of the sum over $2 \le \ell <L$ in the previous formula is at most
\[ \sum_{\ell=2}^{L-1} a_{\ell+1} \left( \frac{a_{\ell+1} a_1}{q_{\ell +1} q_{\ell+1}'}+ \frac{q_2}{q_{\ell+1}} + \sum_{j=2}^{\ell-1} a_{j+1} \frac{2a_1}{q_{j+1} q_{\ell+1}'} + \frac{1}{q_{\ell+1}'} + \frac{1}{q_{\ell+1}} \right) \ll \sum_{\ell=2}^{L-1} \frac{1}{q_{\ell}'} \ll 1 . \]
This finishes the proof of the first claim.

If $b_0(N)=\cdots =b_{k-1}(N)=0$ with some $k \ge 2$, then the terms $\ell \le k-1$ are all zero, and the contribution of the terms $k+1 \le \ell <L$ is similarly seen to be $\sum_{\ell=k+1}^{L-1}1/q_{\ell}' \ll 1/q_{k+1}'$. Finally, the $\ell=k$ term is
\[ (-1)^{k+1} b_k(N) \left( \frac{b_k(N) (q_k' \| q_k' \alpha' \| - q_k \| q_k \alpha \|)}{2} + \frac{\| q_k' \alpha' \| - \| q_k \alpha \|}{2} \right) . \]
As we have seen, with $R=[a_{k+1};a_{k+2}, \ldots, a_L]$ resp.\ $R=[a_{k+1};a_{k+2},\ldots ]$ here
\[ q_k' \| q_k' \alpha' \| - q_k \| q_k \alpha \| = \frac{(-1)^k a_1}{(Rq_k +q_{k-1})(R q_k'+q_{k-1}')} = \frac{(-1)^k a_1}{a_{k+1}^2 q_k q_k'} +O \left( \frac{a_1}{a_{k+1}^3 q_k q_k'} \right) \]
and using \eqref{pell'qell'},
\[ \| q_k' \alpha' \| - \| q_k \alpha \| = \frac{1}{R q_k' +q_{k-1}'} - \frac{1}{R q_k +q_{k-1}} = \frac{a_1 p_k}{a_{k+1} q_k q_k'} + O \left( \frac{a_1 p_k}{a_{k+1}^2 q_k q_k'} \right) , \]
and the second claim follows.
\end{proof}

\subsection{Asymptotics of $h_p$}

We now prove Theorem \ref{asymptoticstheorem} on the asymptotics of $h_p$ after a preparatory lemma.
\begin{lem}\label{aplemma} For any $0<p<\infty$ and any integer $a \ge 1$,
\begin{equation}\label{apfirst}
\sum_{b=0}^{a-1} e^{\frac{pa}{2} \cdot \frac{b}{a} \left( 1-\frac{b}{a} \right)} = \exp \left( \frac{pa}{8} + \frac{1}{2} \log a + O \left( \max \left\{ p, \log \frac{1}{p} \right\} \right) \right)
\end{equation}
and
\begin{equation}\label{apsecond}
\sum_{b=0}^{a-1} e^{- \frac{pa}{2} \cdot \frac{b}{a} \left( 1-\frac{b}{a} \right)} = \exp \bigg( O \left( \max \left\{ p, \log \frac{1}{p} \right\} \right) \bigg)
\end{equation}
with universal implied constants.
\end{lem}

\begin{proof} We start with \eqref{apfirst}. Each term in the sum is at most $e^{pa/8}$, thus comparing the sum to the corresponding integral leads to the upper bound
\[ \sum_{b=0}^{a-1} e^{\frac{pa}{2} \cdot \frac{b}{a} \left( 1-\frac{b}{a} \right)} \le a \int_0^1 e^{\frac{pa}{2} x(1-x)} \, \mathrm{d}x + e^{\frac{pa}{8}} \le a e^{\frac{pa}{8}} \int_{-\infty}^{\infty} e^{-\frac{pa}{2} (x-1/2)^2} \, \mathrm{d} x + e^{\frac{pa}{8}} = \left( \sqrt{\frac{2 \pi a}{p}} +1 \right) e^{\frac{pa}{8}} . \]
Here
\[ \log  \left( \sqrt{\frac{2 \pi a}{p}} +1 \right) \le \frac{1}{2} \log a + O \left( \max \left\{ p, \log \frac{1}{p} \right\} \right) , \]
and the $\le$ part of \eqref{apfirst} follows. Since $e^{\frac{pa}{2}x(1-x)}$ is increasing on $[0,1/2]$, comparing the sum to the corresponding integral leads to the lower bound
\[ \sum_{b=1}^{\lfloor a/2 \rfloor} e^{\frac{pa}{2} \cdot \frac{b}{a} \left( 1-\frac{b}{a} \right)} \ge a \int_0^{\frac{\lfloor a/2 \rfloor}{a}} e^{\frac{pa}{2} x(1-x)} \, \mathrm{d}x \ge a e^{\frac{pa}{8}} \int_{0}^{\frac{a-1}{2a}} e^{- \frac{pa}{2} (x-1/2)^2} \, \mathrm{d}x = \sqrt{\frac{a}{p}} e^{\frac{pa}{8}} \int_{- \frac{\sqrt{pa}}{2}}^{-\frac{\sqrt{p}}{2 \sqrt{a}}} e^{-x^2/2} \, \mathrm{d}x . \]
If $pa \ge 100$ and $p \le 64a$, then $-\sqrt{pa}/2 \le -5$ and $-\sqrt{p}/(2 \sqrt{a}) \ge -4$, thus the previous formula yields
\[ \sum_{b=0}^{a-1} e^{\frac{pa}{2} \cdot \frac{b}{a} \left( 1-\frac{b}{a} \right)} \gg \sqrt{\frac{a}{p}} e^{\frac{pa}{8}}, \]
which suffices for the $\ge$ part of \eqref{apfirst}. If $pa<100$, then simply using the fact that each term is at least 1 yields
\[ \sum_{b=0}^{a-1} e^{\frac{pa}{2} \cdot \frac{b}{a} \left( 1-\frac{b}{a} \right)} \ge a \ge \exp \left( \frac{pa}{8} + \frac{1}{2} \log a - \frac{100}{8} \right) , \]
which again suffices for the $\ge$ part of \eqref{apfirst}. If $p > 64 a$, then it is enough to keep the $b=\lfloor a/2 \rfloor$ term in the sum, yielding
\[ e^{\frac{pa}{2} \cdot \frac{\lfloor a/2 \rfloor}{a} \left( 1-\frac{\lfloor a/2 \rfloor}{a} \right)} \ge e^{\frac{pa}{2} \cdot \frac{a-1}{2a} \left( 1-\frac{a-1}{2a} \right)} = e^{\frac{pa}{8} - \frac{p}{8a}} \ge \exp \left( \frac{pa}{8} + \frac{1}{2} \log a - \frac{1}{2} \log \frac{p}{64} - \frac{p}{8} \right) , \]
which also suffices for the $\ge$ part of \eqref{apfirst}. This finishes the proof of \eqref{apfirst}.

We now prove \eqref{apsecond}. Keeping only the term $b=0$ gives the trivial lower bound $1$. Since each term is at most $1$, comparing the sum to the corresponding integral leads to the upper bound
\[ \begin{split} 1 \le \sum_{b=0}^{a-1} e^{- \frac{pa}{2} \cdot \frac{b}{a} \left( 1-\frac{b}{a} \right)} \le a \int_{0}^{1} e^{- \frac{pa}{2} x(1-x)} \, \mathrm{d}x +1 &= a e^{-\frac{pa}{8}} \int_{-1/2}^{1/2} e^{\frac{pa}{2} x^2} \, \mathrm{d} x +1 \\ &= \frac{8}{p} \sqrt{\frac{pa}{8}} e^{- \frac{pa}{8}} \int_{0}^{\sqrt{\frac{pa}{8}}} e^{x^2} \, \mathrm{d}x +1 \\ &\ll \frac{1}{p}+1 . \end{split} \]
In the last step we used the fact that $\sup_{y \ge 0} y e^{-y^2} \int_{0}^{y} e^{x^2} \, \mathrm{d}x < \infty$. This establishes \eqref{apsecond}.
\end{proof}

\begin{proof}[Proof of Theorem \ref{asymptoticstheorem}] It will be enough to prove the theorem for finite $p$. The claim for $p=\pm \infty$ then follows from taking the limit as $p \to \pm \infty$.

Let $r=[0;a_1,a_2,\ldots, a_L]$ be rational with denominator $q$ and convergents $p_{\ell}/q_{\ell}=[0;a_1,a_2,\ldots, a_{\ell}]$. Let $r'=T^2 r = [0;a_3,a_4,\ldots, a_L]$ with denominator $q'$ and convergents $p_{\ell}'/q_{\ell}'=[0;a_3,a_4,\ldots, a_{\ell}]$, $3 \le \ell \le L$, and $p_2'=0$, $q_2'=1$.

Fix integers $0 \le b_0 <a_1$ and $0 \le b_1 \le a_2$ such that $b_1=a_2$ implies $b_0=0$. Observe that the map $N \mapsto N-q_2$ is an injection from
\[ \left\{ 0 \le N <q \, : \, b_0(N)=b_0, \,\, b_1(N)=0, \,\, b_2(N)=a_3 \right\} \]
to
\[ \left\{ 0 \le N <q \, : \, b_0(N)=b_0, \,\, b_1(N)=0, \,\, b_2(N)=a_3-1 \right\} . \]
Two applications of Lemma \ref{localoptimumlemma} (to $N$ and $N-q_2$, with $k=1$) shows that $S_N(r)=S_{N-q_2}(r)+O(1)$, therefore
\[ \sum_{\substack{0 \le N <q \\ b_0(N)=b_0, \,\, b_1(N)=0, \,\, b_2(N)=a_3}} e^{p S_N(r)} \le \exp (O(|p|)) \sum_{\substack{0 \le N <q \\ b_0(N)=b_0, \,\, b_1(N)=0, \,\, b_2(N)=a_3-1}} e^{p S_N(r)} . \]
In particular,
\[ \sum_{\substack{0 \le N <q \\ b_0(N)=b_0, \,\, b_1(N)=b_1}} e^{p S_N(r)} = \exp \left( O \left( \max \{ |p|, 1\} \right) \right) \sum_{\substack{0 \le N <q \\ b_0(N)=b_0, \,\, b_1(N)=b_1, \,\, b_2(N)<a_3}} e^{p S_N(r)} , \]
the formula being trivial for $b_1 \neq 0$, as in that case the two sums are identical.

The ``matching'' map $N \to N'$ introduced in Section \ref{matchingsection} is a bijection
\[ \left\{ 0 \le N <q \, : \, b_0(N)=b_0, \,\, b_1(N)=b_1, \,\, b_2(N)<a_3 \right\} \to [0,q') , \]
and by Lemma \ref{matchinglemma},
\[ \begin{split} S_N(r)-S_{N'}(r') = &-b_0 \left( \frac{1-b_0 q_0 \| q_0 r \|}{2} - \frac{\| q_0 r \|}{2} \right) + b_1 \left( \frac{1-b_1 q_1 \| q_1 r \|}{2} - \| q_1 r \| b_0 q_0 - \frac{\| q_1 r \|}{2} \right) + O(1) \\ = &-b_0 \frac{1-b_0/a_1}{2} + b_1 \frac{1-b_1/a_2}{2} + O(1) . \end{split} \]
Consequently,
\[ \sum_{\substack{0 \le N < q \\ b_0(N)=b_0, \,\, b_1(N)=b_1}} e^{p S_N(r)} = \exp \left( -p b_0 \frac{1-b_0/a_1}{2} + p b_1 \frac{1-b_1/a_2}{2} + O( \max \{ |p| , 1 \} ) \right) \sum_{0 \le N <q'} e^{p S_N (r')} . \]
We now sum over all possible values of $b_0, b_1$, and apply Lemma \ref{aplemma} to deduce
\[ \begin{split} \sum_{0 \le N <q} e^{p S_N(r)} &= \left( 1+ \sum_{b_0=0}^{a_1-1} e^{-p b_0 \frac{1-b_0/a_1}{2}} \sum_{b_1=0}^{a_2-1} e^{p b_1 \frac{1-b_1/a_2}{2}} \right) \exp (O( \max \{ |p|, 1 \} )) \sum_{0 \le N <q'} e^{p S_N (r')} \\ &= \exp \left( \frac{|p| a_{\varepsilon_p}}{8} + \frac{1}{2} \log a_{\varepsilon_p} +O \left( \max \left\{ |p|, \log \frac{1}{|p|} \right\} \right) \right) \sum_{0 \le N <q'} e^{p S_N (r')} . \end{split} \]
By the definition of $h_p$, this means that
\[ h_p(r) = \mathrm{sgn} (p) \frac{a_{\varepsilon_p}}{8} + \frac{1}{2p} \log a_{\varepsilon_p} + O \left( \max \left\{ 1, \frac{1}{|p|} \log \frac{1}{|p|} \right\} \right) , \]
which is an equivalent form of the claim.
\end{proof}

\subsection{Continuity of $h_p$ at irrationals}

We now prove Theorem \ref{continuitytheorem} in a quantitative form, establishing an estimate for the modulus of continuity as well. Fix an irrational $\alpha \in (0,1)$ with continued fraction expansion $\alpha = [0;a_1,a_2,\ldots]$ and convergents $p_k/q_k=[0;a_1,a_2,\ldots, a_k]$. Let
\[ I_{k+1} = \left\{ [0;c_1,c_2,\ldots] \, : \, c_j=a_j \textrm{ for all } 1 \le j \le k+1 \right\} \]
denote the set of real numbers in $(0,1)$ whose first $k+1$ partial quotients are identical to those of $\alpha$. Recall that $I_{k+1} \subset (0,1)$ is an interval with rational endpoints; in particular, $\alpha \in \mathrm{int} \, I_{k+1}$.
\begin{thm}\label{modulusofcontinuitytheorem} Let $- \infty \le p \le \infty$, $p \neq 0$, and let $k \ge 2$ be an integer such that $a_{k+1} \ge A \max \{ 1, \frac{1}{|p|} \log \frac{1}{|p|} \}$ with a large universal constant $A>1$, and $k+1 \equiv \varepsilon_p \pmod{2}$. Then
\[ \sup_{r \in I_{k+1} \cap \mathbb{Q}} h_p(r) - \inf_{r \in I_{k+1} \cap \mathbb{Q}} h_p(r) \ll \frac{a_1 a_2}{q_k} + \sqrt{\frac{\log a_{k+1}}{\min \{1,|p| \} a_{k+1}}} \]
with a universal implied constant.
\end{thm}
\noindent In particular, if $\sup_{k \in \mathbb{N}} a_{2k+\varepsilon_p}=\infty$, then
\[ \liminf_{\substack{k \to \infty \\ k+1 \equiv \varepsilon_p \pmod{2}}} \left( \frac{a_1 a_2}{q_k} +\sqrt{\frac{\log a_{k+1}}{\min \{1,|p| \} a_{k+1}}}  \right) =0 , \]
and consequently $\lim_{r \to \alpha} h_p(r)$ exists and is finite by the Cauchy criterion. This proves Theorem \ref{continuitytheorem}.

\begin{proof}[Proof of Theorem \ref{modulusofcontinuitytheorem}] We only give a detailed proof for finite $p$, as the proof for $p=\pm \infty$ is entirely analogous. Let $\alpha'=T^2 \alpha = [0;a_3,a_4,\ldots]$, and let $p_{\ell}'/q_{\ell}'=[0;a_3,a_4,\ldots ]$, $\ell \ge 3$ and $p_2'=0$, $q_2'=1$ denote its convergents.

Let $r \in I_{k+1} \cap \mathbb{Q}$ be arbitrary with denominator $q$, continued fraction expansion $r=[0;c_1,c_2,\ldots, c_L]$ and convergents $\bar{p}_{\ell}/\bar{q}_{\ell}=[0;c_1,c_2,\ldots, c_{\ell}]$. Let $r'=T^2r=[0;c_3,c_4,\ldots, c_L]$ with denominator $q'$, and convergents $\bar{p}_{\ell}' / \bar{q}_{\ell}'=[0;c_3,c_4,\ldots, c_{\ell}]$, $3 \le \ell \le L$ and $\bar{p}_2'=0$, $\bar{q}_2'=1$. By construction, we have $\bar{p}_{\ell}/\bar{q}_{\ell} = p_{\ell}/q_{\ell}$ for all $0 \le \ell \le k+1$, and $\bar{p}_{\ell}'/\bar{q}_{\ell}' = p_{\ell}'/q_{\ell}'$ for all $2 \le \ell \le k+1$.

An application of Lemma \ref{factorlemma} to $r$ resp.\ $r'$ with $K=L$ yields
\[ \begin{split} \Bigg( \sum_{0 \le N <q} e^{p S_N(r)} \Bigg)^{1/p} = &\left( 1+O \left( \sqrt{\frac{\log a_{k+1}}{\min \{1,|p| \} a_{k+1}}} \right) \right) \Bigg( \sum_{0 \le N <q_k} e^{p (S_N(p_k/q_k) + (-1)^k N/(2q_k))} \Bigg)^{1/p} \\ &\times \Bigg( \sum_{\substack{0 \le N < q \\ b_0(N) = \cdots = b_{k-1}(N) =0 \\ |b_k(N) - a_{k+1}/2| \le \max \{ 10,10/\sqrt{|p|} \} \sqrt{a_{k+1} \log a_{k+1}}}} e^{p S_N(r)} \Bigg)^{1/p} \end{split} \]
resp.
\[ \begin{split} \Bigg( \sum_{0 \le N <q'} e^{p S_N(r')} \Bigg)^{1/p} = &\left( 1+O \left( \sqrt{\frac{\log a_{k+1}}{\min \{1,|p| \} a_{k+1}}} \right) \right) \Bigg( \sum_{0 \le N <q_k'} e^{p (S_N(p_k'/q_k') + (-1)^k N/(2q_k'))} \Bigg)^{1/p} \\ &\times \Bigg( \sum_{\substack{0 \le N < q' \\ b_2'(N) = \cdots = b_{k-1}'(N) =0 \\ |b_k'(N) - a_{k+1}/2| \le \max \{ 10,10/\sqrt{|p|} \} \sqrt{a_{k+1} \log a_{k+1}}}} e^{p S_N(r')} \Bigg)^{1/p} . \end{split} \]
Here $b_{\ell}(N)$ resp.\ $b_{\ell}'(N)$ denote the digits in the Ostrowski expansion with respect to $r$ resp.\ $r'$. Consequently,
\begin{equation}\label{hpestimate}
\begin{split} h_p(r)=\log \frac{J_p(r)}{J_p(r')} = &Z_{p,k}(\alpha) + \frac{1}{p} \log \frac{\displaystyle{\sum_{\substack{0 \le N < q \\ b_0(N) = \cdots = b_{k-1}(N) =0 \\ |b_k(N) - a_{k+1}/2| \le \max \{ 10,10/\sqrt{|p|} \} \sqrt{a_{k+1} \log a_{k+1}}}} e^{p S_N(r)}}}{\displaystyle{\sum_{\substack{0 \le N < q' \\ b_2'(N) = \cdots = b_{k-1}'(N) =0 \\ |b_k'(N) - a_{k+1}/2| \le \max \{ 10,10/\sqrt{|p|} \} \sqrt{a_{k+1} \log a_{k+1}}}} e^{p S_N(r')}}} \\ &+ O \left( \sqrt{\frac{\log a_{k+1}}{\min \{1,|p| \} a_{k+1}}} \right) , \end{split}
\end{equation}
with the crucial observation that
\[ Z_{p,k}(\alpha) := \frac{1}{p} \log \frac{\displaystyle{\sum_{0 \le N <q_k} e^{p (S_N(p_k/q_k) + (-1)^k N/(2q_k))}}}{\displaystyle{\sum_{0 \le N <q_k'} e^{p (S_N(p_k'/q_k') + (-1)^k N/(2q_k'))}}} \]
depends only on $\alpha$, but not on $r$.

The ``matching'' map $N \mapsto N'$ introduced in Section \ref{matchingsection} is a bijection from the set
\[ \left\{ 0 \le N<q \, : \, \begin{array}{c} b_0(N) = \cdots = b_{k-1}(N)=0, \\ |b_k(N) - a_{k+1}/2| \le \max \{ 10,10/\sqrt{|p|} \} \sqrt{a_{k+1} \log a_{k+1}} \end{array} \right\} \]
to the set
\[ \left\{ 0 \le N<q' \, : \, \begin{array}{c} b_2'(N) = \cdots = b_{k-1}'(N)=0, \\ |b_k'(N) - a_{k+1}/2| \le \max \{ 10,10/\sqrt{|p|} \} \sqrt{a_{k+1} \log a_{k+1}} \end{array} \right\} , \]
and by Lemma \ref{matchinglemma}, $| S_N(r)-S_{N'}(r') | \ll 1/q_k' \ll a_1 a_2 /q_k$. Hence
\[ \frac{\displaystyle{\sum_{\substack{0 \le N < q \\ b_0(N) = \cdots = b_{k-1}(N) =0 \\ |b_k(N) - a_{k+1}/2| \le \max \{ 10,10/\sqrt{|p|} \} \sqrt{a_{k+1} \log a_{k+1}}}} e^{p S_N(r)}}}{\displaystyle{\sum_{\substack{0 \le N < q' \\ b_2'(N) = \cdots = b_{k-1}'(N) =0 \\ |b_k'(N) - a_{k+1}/2| \le \max \{ 10,10/\sqrt{|p|} \} \sqrt{a_{k+1} \log a_{k+1}}}} e^{p S_N(r')}}} = \exp \left( O \left( |p| \frac{a_1 a_2}{q_k} \right) \right) , \]
and \eqref{hpestimate} leads to
\[ h_p (r) = Z_{p,k}(\alpha) + O \left( \frac{a_1 a_2}{q_k} + \sqrt{\frac{\log a_{k+1}}{\min \{1,|p| \} a_{k+1}}} \right) \qquad \textrm{uniformly in } r \in I_{k+1} \cap \mathbb{Q} . \]
This establishes the desired upper bound to the oscillation of $h_p$ on the set $I_{k+1} \cap \mathbb{Q}$.
\end{proof}

\subsection{One-sided limit of $h_p$ at rationals}

\begin{proof}[Proof of Theorem \ref{onesidedlimittheorem}] We only give a detailed proof for finite $p$, as the proof for $p=\pm \infty$ is entirely analogous.

Fix a reduced rational $a/q \in (0,1)$. It has exactly two continued fraction expansions, one of even length and one of odd length. Consider thus the expansion $a/q=[0;a_1,a_2,\ldots, a_s]$ with odd $s \ge 3$ if $p>0$, and even $s \ge 2$ if $p<0$, and let $p_k/q_k=[0;a_1,a_2,\ldots, a_k]$ denote its convergents. In particular, $s+1 \equiv \varepsilon_p \pmod{2}$. Let $I(n)$ be the set of all reals of the form $[0;a_1,a_2,\ldots, a_s,m,\ldots]$ with $m \ge n$. Note that $I(n)$ is an interval with endpoints $(p_s n+p_{s-1})/(q_s n +q_{s-1})$ and $p_s/q_s=a/q$. The choice of the parity of $s$ implies that $I(n)=[a/q-\kappa_n, a/q)$ is a left-hand neighborhood if $p>0$, whereas $I(n)=(a/q,a/q+\kappa_n]$ is a right-hand neighborhood if $p<0$, of length $\kappa_n=1/(q_s^2 n+q_{s-1}q_s)$. It will thus be enough to prove that $\sup_{r \in I(n) \cap \mathbb{Q}} |h_p(r)-W_p(a/q)| \to 0$ as $n \to \infty$.

Now let $n>A \max \{ 1,\frac{1}{|p|} \log \frac{1}{|p|} \}$ with a large universal constant $A>1$, and let $r \in I(n) \cap \mathbb{Q}$ be arbitrary. The continued fraction of $r$ is thus of the form $r=[0;a_1,a_2,\ldots, a_L]$ with $L \ge s+1\ge 3$ and $a_{s+1} \ge n$. In particular, the convergents $p_k/q_k$, $0 \le k \le L$ to $r$ coincide with those to $a/q$ for $0 \le k \le s$. Let $r'=T^2 r=[0;a_3,\ldots, a_L]$ with convergents $p_k'/q_k'=[0;a_3,\ldots, a_k]$, $3 \le k \le L$ and $p_2'=0$, $q_2'=1$. Then $a'/q'=T^2 (a/q)=[0;a_3,\ldots, a_s]$ has the same convergents for $2 \le k \le s$.

Following the steps in the proof of Theorem \ref{modulusofcontinuitytheorem} leading up to \eqref{hpestimate} (with $k=s$), we deduce
\[ \begin{split} h_p(r) = &\frac{1}{p} \log \frac{\displaystyle{\sum_{0 \le N <q} e^{p (S_N(a/q) -\mathrm{sgn}(p) N/(2q))}}}{\displaystyle{\sum_{0 \le N <q'} e^{p (S_N(a'/q') -\mathrm{sgn}(p) N/(2q'))}}}\\ &+ \frac{1}{p} \log \frac{\displaystyle{\sum_{\substack{0 \le N < q_L \\ b_0(N) = \cdots = b_{s-1}(N) =0 \\ |b_s(N) - a_{s+1}/2| \le \max \{ 10,10/\sqrt{|p|} \} \sqrt{a_{s+1} \log a_{s+1}}}} e^{p S_N(r)}}}{\displaystyle{\sum_{\substack{0 \le N < q_L' \\ b_2'(N) = \cdots = b_{s-1}'(N) =0 \\ |b_s'(N) - a_{s+1}/2| \le \max \{ 10,10/\sqrt{|p|} \} \sqrt{a_{s+1} \log a_{s+1}}}} e^{p S_N(r')}}} + O \left( \sqrt{\frac{\log n}{\min \{1,|p| \} n}} \right) . \end{split} \]
Here $b_{\ell}(N)$ resp.\ $b_{\ell}'(N)$ denote the digits in the Ostrowski expansion with respect to $r$ resp.\ $r'$. The first term in the previous formula depends only on $a/q$ but not on $r$.

It remains to estimate the second term. The ``matching'' map $N \mapsto N'$ introduced in Section \ref{matchingsection} is a bijection from the set
\[ \left\{ 0 \le N<q_L \, : \, \begin{array}{c} b_0(N) = \cdots = b_{s-1}(N)=0, \\ |b_s(N) - a_{s+1}/2| \le \max \{ 10,10/\sqrt{|p|} \} \sqrt{a_{s+1} \log a_{s+1}} \end{array} \right\} \]
to the set
\[ \left\{ 0 \le N<q_L' \, : \, \begin{array}{c} b_2'(N) = \cdots = b_{s-1}'(N)=0, \\ |b_s'(N) - a_{s+1}/2| \le \max \{ 10,10/\sqrt{|p|} \} \sqrt{a_{s+1} \log a_{s+1}} \end{array} \right\} . \]
By Lemma \ref{matchinglemma}, for all such $N$,
\[ \begin{split} S_N(r) - S_{N'}(r') &= a_1 \frac{\mathrm{sgn}(p) p_s b_s(N)/a_{s+1} - (b_s(N)/a_{s+1})^2}{2 q_s q_s'} + O \left( \frac{1}{q_{s+1}'} \right) \\ &= \lfloor q/a \rfloor \frac{\mathrm{sgn}(p) a /2 - 1/4}{2 q q'} + O \left( \sqrt{\frac{\log n}{\min \{ 1,|p| \} n}} \right) , \end{split} \]
consequently
\[ \begin{split} \frac{1}{p} \log \frac{\displaystyle{\sum_{\substack{0 \le N < q_L \\ b_0(N) = \cdots = b_{s-1}(N) =0 \\ |b_s(N) - a_{s+1}/2| \le \max \{ 10,10/\sqrt{|p|} \} \sqrt{a_{s+1} \log a_{s+1}}}} e^{p S_N(r)}}}{\displaystyle{\sum_{\substack{0 \le N < q_L' \\ b_2'(N) = \cdots = b_{s-1}'(N) =0 \\ |b_s'(N) - a_{s+1}/2| \le \max \{ 10,10/\sqrt{|p|} \} \sqrt{a_{s+1} \log a_{s+1}}}} e^{p S_N(r')}}} = &\lfloor q/a \rfloor \frac{\mathrm{sgn}(p) a /2 - 1/4}{2 q q'} \\ &+ O \left( \sqrt{\frac{\log n}{\min \{ 1,|p| \} n}} \right) . \end{split} \]
Hence
\[ h_p(r) = W_p(a/q) + O \left( \sqrt{\frac{\log n}{\min \{ 1,|p| \} n}} \right) \qquad \textrm{uniformly in } r \in I(n), \]
and the desired limit relation follows.
\end{proof}

\section{Quadratic irrationals}\label{proofquadraticsection}

Fix a quadratic irrational $\alpha$ and a parameter $-\infty \le p \le \infty$, $p \neq 0$. Throughout this section, constants and implied constants may depend on $\alpha$.

Let us write the continued fraction expansion in the form $\alpha=[a_0;a_1,\ldots, a_s, \overline{a_{s+1}, \ldots, a_{s+m}}]$, where the overline denotes the period. We can always choose the period length $m$ to be even, although it might not be the shortest possible period. This choice is convenient because $S_N(\alpha)$ is odd in the variable $\alpha$, cf.\ the alternating factor $(-1)^{\ell+1}$ in Ostrowski's explicit formula in Lemma \ref{ostrowskilemma}. Solving the recursions with periodic coefficients gives that for any $k \ge 0$ and $1 \le r \le m$,
\begin{equation}\label{pkqkrecursion}
q_{s+km+r} = E_r \eta^k+F_r \eta^{-k} \qquad \textrm{and} \qquad \| q_{s+km+r} \alpha \|= G_r \eta^{-k}
\end{equation}
with some explicitly computable constants $\eta>1$, $E_r,G_r>0$ and $F_r \in \mathbb{R}$, $1 \le r \le m$ \cite[Eq.\ (28)]{AB2}.

The following lemma states that shifting the digits in the Ostrowski expansion by full periods has a negligible effect.
\begin{lem}\label{shiftlemma} Let $0 \le N <q_{s+km}$ be an integer with Ostrowski expansion $N=\sum_{\ell =s}^{s+km-1} b_{\ell}(N) q_{\ell}$. Let $i \ge 1$ be an integer, and set $N^{(i)}=\sum_{\ell=s+im}^{s+(i+k)m-1} b_{\ell-im}(N)q_{\ell}$. Then $|S_N(\alpha) - S_{N^{(i)}}(\alpha)| \ll 1$.
\end{lem}

\begin{proof} Note that the shift results in a legitimate Ostrowski expansion for $N^{(i)}$, that is, $b_{\ell}(N^{(i)})=b_{\ell-im}(N)$ for all $s+im \le \ell \le s+(i+k)m-1$. Applying Ostrowski's explicit formula in Lemma \ref{ostrowskilemma} to $N$ and $N^{(i)}$ thus yields
\[ \begin{split} S_N(\alpha) - S_{N^{(i)}}(\alpha) = \sum_{\ell=s}^{s+km-1} (-1)^{\ell+1} b_{\ell}(N) \bigg( &\frac{b_{\ell}(N) ( q_{\ell+im} \| q_{\ell+im} \alpha \| - q_{\ell} \| q_{\ell} \alpha \| )}{2} \\ &+ \sum_{j=s}^{\ell-1} b_j(N) (q_{j+im} \| q_{\ell+im} \alpha \| -q_j \| q_{\ell} \alpha \|) \\ &+\frac{\| q_{\ell+im} \alpha \| - \| q_{\ell} \alpha \|}{2} \bigg) . \end{split} \]
Formula \eqref{pkqkrecursion} shows that here $q_{j+im} \| q_{\ell+im} \alpha \| - q_j \| q_{\ell} \alpha \| = O(\eta^{-(j+\ell)/m})$ for all $s \le j \le \ell$, and the claim follows.
\end{proof}

We now show that $\log J_{p,M}(\alpha)$ with $M=q_{s+km}$ is approximately additive in $k$.
\begin{lem}\label{additivelemma} For any integers $i,k \ge 1$,
\[ \log J_{p,q_{s+(i+k)m}}(\alpha) = \log J_{p,q_{s+im}}(\alpha) + \log J_{p,q_{s+km}}(\alpha) +O(\max \{ 1,1/|p| \}) . \]
\end{lem}

\begin{proof} It will be enough to prove the lemma for finite $p$. The claim for $p=\pm \infty$ then follows from taking the limit as $p \to \pm \infty$.

Note that each individual term in Ostrowski's explicit formula in Lemma \ref{ostrowskilemma} is $O(1)$. In particular, $S_N(\alpha)=O(1)$ whenever $N$ has $O(1)$ nonzero digits in its Ostrowski expansion. More generally, changing a single Ostrowski digit of $N$ changes the value of $S_N(\alpha)$ by $O(1)$.

Let $c_k=\sum_{0 \le N<q_{s+km}} e^{p S_N(\alpha)}$, $k \ge 1$. Observe that the map $[0,q_{s+(k+1)m}) \to [0,q_{s+km})$, $N=\sum_{\ell=0}^{s+(k+1)m-1} b_{\ell}(N) q_{\ell} \mapsto N^-=\sum_{\ell=0}^{s+km-1} b_{\ell}(N) q_{\ell}$ has the property that each value is attained $O(1)$ times. Since $N^-$ is obtained from $N$ by deleting a single Ostrowski digit, we have $S_{N^-}(\alpha) = S_N (\alpha)+O(1)$. Hence for all $k \ge 1$,
\begin{equation}\label{ck+1ck}
c_{k+1} \le e^{O(|p|)} \sum_{0 \le N<q_{s+(k+1)m}} e^{pS_{N^-} (\alpha)} \le e^{O(\max\{ |p|, 1\})} c_k.
\end{equation}

Now fix $i,k \ge 1$. Let $0\le N'<q_{s+im}$ and $0 \le N''<q_{s+km}$ be integers with Ostrowski expansions $N'=\sum_{\ell =0}^{s+im-1} b_{\ell}(N') q_{\ell}$ and $N''=\sum_{\ell =0}^{s+km-1} b_{\ell}(N'') q_{\ell}$. Define $0 \le N<q_{s+(i+k)m}$, $N=\sum_{\ell=0}^{s+(i+k)m-1} b_{\ell}(N)q_{\ell}$ as
\[ b_{\ell}(N) = \left\{ \begin{array}{ll} b_{\ell}(N') & \textrm{if } 0 \le \ell \le s+im-1, \\ 0 & \textrm{if } s+im \le \ell \le s+(i+1)m-1, \\ b_{\ell-(i+1)m}(N'') & \textrm{if } s+(i+1)m \le \ell \le s+(i+k+1)m-1 . \end{array} \right. \]
Note that the block of zeroes in the middle ensures that the extra rule of Ostrowski expansions ($b_{\ell+1}(N)=a_{\ell+2}$ implies $b_{\ell}(N)=0$) is satisfied. The map $[0,q_{s+km}) \times [0,q_{s+im}) \to [0,q_{s+(i+k+1)m})$, $(N',N'') \mapsto N$ is injective. Deleting the first $s$ Ostrowski digits of $N''$, and then applying Lemmas \ref{shiftlemma} and \ref{SNlemma} shows that $S_N(\alpha) = S_{N'}(\alpha)+S_{N''}(\alpha)+O(1)$. Using \eqref{ck+1ck} as well thus leads to
\begin{equation}\label{cick}
c_i c_k = \sum_{\substack{0 \le N' <q_{s+im} \\ 0 \le N'' < q_{s+km}}} e^{p (S_{N'}(\alpha) + S_{N''}(\alpha))} \le e^{O(|p|)} c_{i+k+1} \le e^{O(\max \{ |p|,1 \} )} c_{i+k} .
\end{equation}

Next, for any integer $0 \le N < q_{s+(i+k)m}$ with Ostrowski expansion $N=\sum_{\ell=0}^{s+(i+k)m-1} b_{\ell}(N) q_{\ell}$ define $N_1=\sum_{\ell=0}^{s+im-1} b_{\ell}(N) q_{\ell}$ and $N_2=\sum_{\ell=s}^{s+km-1} b_{\ell+im}(N) q_{\ell}$. Note that, with the notation of Lemma \ref{shiftlemma}, $N=N_1+N_2^{(i)}$, hence Lemmas \ref{SNlemma} and \ref{shiftlemma} give $S_N(\alpha) = S_{N_1}(\alpha) + S_{N_2}(\alpha)+O(1)$. Observe that the map $[0,q_{s+(i+k)m}) \to [0,q_{s+im}) \times [0,q_{s+km})$, $N \mapsto (N_1,N_2)$ is injective, thus
\[ c_{i+k} \le e^{O(|p|)} \sum_{\substack{0 \le N_1 < q_{s+im} \\ 0 \le N_2 < q_{s+km}}} e^{p (S_{N_1}(\alpha) + S_{N_2}(\alpha))} = e^{O(|p|)} c_i c_k . \]
The previous formula together with \eqref{cick} show that $c_{i+k} = e^{O(\max \{ |p|,1 \})} c_i c_k$, and the claim follows.
\end{proof}

\begin{proof}[Proof of Theorem \ref{quadratictheorem}] By Lemma \ref{additivelemma}, there exists a constant $K=O(\max \{ 1,1/|p| \} )$ such that the sequence $\log J_{p,q_{s+km}} (\alpha) +K$ resp.\ $\log J_{p,q_{s+km}} (\alpha) -K$ is subadditive resp.\ superadditive in $k$. An application of the subadditive lemma of Fekete then shows that the sequence $k^{-1}\log J_{p,q_{s+km}} (\alpha)$ is convergent, and denoting its limit by $C'_p(\alpha)$,
\[ C'_p(\alpha) = \inf_{k \ge 1} \frac{\log J_{p,q_{s+km}} (\alpha) +K}{k} = \sup_{k \ge 1} \frac{\log J_{p,q_{s+km}} (\alpha) -K}{k} . \]
In particular, $\log J_{p,q_{s+km}} (\alpha)=C'_p(\alpha) k +O(\max \{ 1,1/|p| \} )$.

Given an arbitrary integer $q_{s+km} \le M < q_{s+(k+1)m}$, we have
\[ \log J_{p,q_{s+km}}(\alpha) \le \log J_{p,M}(\alpha) \le \log J_{p,q_{s+(k+1)m}} (\alpha) \]
if $p>0$, and the reverse inequalities hold if $p<0$. Formula \eqref{pkqkrecursion} shows that $\log q_{s+km} = (\log \eta) k+O(1)$, hence
\[ \log J_{p,M}(\alpha) = C'_p(\alpha) k +O(\max \{ 1,1/|p| \} ) = \frac{C'_p(\alpha)}{\log \eta} \log M + O(\max \{ 1,1/|p| \} ) . \]
Thus $C_p(\alpha) = C'_p(\alpha) / \log \eta$ satisfies the claim of the theorem.
\end{proof}

\section{Proof of the limit laws}\label{prooflimitlawsection}

For any $r \in (0,1) \cap \mathbb{Q}$, define
\begin{equation}\label{gp}
g_p (r) = h_p(r) - \left\{ \begin{array}{ll} \mathds{1}_{\{ Tr \neq 0 \}} \frac{1}{8} \lfloor \frac{1}{Tr} \rfloor & \textrm{if } p>0, \\ - \frac{1}{8} \lfloor \frac{1}{r} \rfloor & \textrm{if } p<0 . \end{array} \right.
\end{equation}
By Theorem \ref{continuitytheorem}, $g_p$ can be extended to an a.e.\ continuous function on $[0,1]$, which we simply denote by $g_p$ as well. By Theorem \ref{asymptoticstheorem}, we have $|g_p(x)| \le c(1+\log (1/Tx))$ if $p>0$, and $|g_p(x)| \le c(1+\log (1/x))$ if $p<0$ with a large constant $c>0$ depending only on $p$.
\begin{lem}\label{gplemma} For any $\varepsilon>0$, there exist a constant $\delta_p>0$ and functions $g_p^{\pm}$ on $[0,1]$ with the following properties.
\begin{enumerate}
\item[(i)] $g_p^- \le g_p \le g_p^+$ on $[0,1]$, and $\int_0^1 (g_p^+(x)-g_p^-(x)) \, \mathrm{d}x < \varepsilon$.
\item[(ii)] If $p>0$, then for all $n \in \mathbb{N}$, the functions $g_p^{\pm}$ are smooth on $(\frac{1}{n+1}, \frac{1}{n})$, and $g_p^{\pm} (x)=\pm 2c \log (1/Tx)$ for all $x \in (\frac{1}{n+1}, \frac{1}{n}) \cap (\frac{1}{n}-\delta_p, \frac{1}{n})$.
\item[(iii)] If $p<0$, then the functions $g_p^{\pm}$ are smooth on $(0,1)$, and $g_p^{\pm}(x)=\pm 2c \log (1/x)$ for all $x \in (0,\delta_p)$.
\end{enumerate}
\end{lem}

\begin{proof} Fix $\varepsilon>0$. Assume first, that $p>0$, and let $\delta_p>0$ be a small constant to be chosen. If $n$ is large enough so that $\frac{1}{n}-\delta_p \le \frac{1}{n+1}$, then we are forced to define $g_p^{\pm}(x)=\pm 2c \log (1/Tx)$ for $x \in (\frac{1}{n+1},\frac{1}{n})$. Now let $n$ be such that $\frac{1}{n}-\delta_p>\frac{1}{n+1}$. Since $g_p$ is bounded and a.e.\ continuous, and consequently Riemann integrable on $[\frac{1}{n+1},\frac{1}{n}-\delta_p]$, we can approximate $g_p$ pointwise from above and from below by step functions, and extend them to $(\frac{1}{n}-\delta_p, \frac{1}{n})$ as $\pm 2c \log (1/Tx)$. By choosing $\delta_p$ small enough, we can ensure that these piecewise defined upper and lower approximating functions are $\varepsilon$-close to each other in $L^1$. Next, we approximate the piecewise defined functions from above and from below by smooth functions which are still $\varepsilon$-close to each other in $L^1$.

The construction for $p<0$ is similar. We first approximate $g_p$ from above and from below by step functions on $[\delta_p,1]$, and extend them as $\pm 2c \log (1/x)$ on $(0,\delta_p)$. Then we approximate these piecewise defined functions from above and from below by smooth functions.
\end{proof}

The following lemma will play a role in the proof of the limit laws for both random rationals and random reals.
\begin{lem}\label{It1t2lemma} For any $t_1, t_2 \in (-1/2,1/2)$,
\[ \begin{split} \int_0^1 \frac{e^{i (t_1 \lfloor 1/Tx \rfloor + t_2 \lfloor 1/x \rfloor )}-1}{1+x} \, \mathrm{d} x = & -\frac{\pi}{2} |t_1| - i \gamma t_1 - i t_1 \log |t_1| - \frac{\pi}{2} |t_2| - i \gamma t_2 - i t_2 \log |t_2| \\ &+O \left( t_1^2 \log \frac{1}{|t_1|} + t_2^2 \log \frac{1}{|t_2|} + |t_1 t_2| \log \frac{1}{|t_1|} \log \frac{1}{|t_2|} \right) \end{split} \]
with a universal implied constant.
\end{lem}

\begin{proof} Let $I(t_1,t_2)$ denote the integral in the claim. Applying the substitution $x \mapsto 1/x$ twice leads to
\[ \begin{split} I(t_1,t_2) &= \int_1^{\infty} \frac{e^{i (t_1 \lfloor 1/\{x\} \rfloor + t_2 \lfloor x \rfloor )}-1}{x(x+1)} \, \mathrm{d} x = \sum_{n=1}^{\infty} \int_0^1 \frac{e^{i (t_1 \lfloor 1/x \rfloor + t_2 n )}-1}{(x+n)(x+n+1)} \, \mathrm{d} x \\ &= \sum_{n=1}^{\infty} \int_1^{\infty} \frac{e^{i (t_1 \lfloor x \rfloor + t_2 n)}-1}{(nx+1)((n+1)x+1)} \, \mathrm{d} x = \sum_{n,m=1}^{\infty} \int_0^1 \frac{e^{i(t_1 m + t_2 n)}-1}{(n(x+m)+1)((n+1)(x+m)+1)} \, \mathrm{d} x \\ &= \sum_{n,m=1}^{\infty} \left( e^{i(t_1 m + t_2 n)} -1 \right) \log \frac{((n+1)(m+1)+1)(nm+1)}{((n+1)m+1)((m+1)n+1)} . \end{split} \]
Here
\[ \begin{split} \log \frac{((n+1)(m+1)+1)(nm+1)}{((n+1)m+1)((m+1)n+1)} &= \log \left( 1+\frac{1}{n^2 m^2 + n^2 m + n m^2 + 3nm +n+m+1} \right) \\ &= \frac{1}{n^2 m^2 + n^2 m + n m^2 + 3nm +n+m+1} + O \left( \frac{1}{n^4 m^4} \right) \\ &= \frac{1}{n(n+1)m(m+1)} + O \left( \frac{1}{n^3 m^3} \right) . \end{split} \]
Letting
\[ R_{n,m} =  \log \frac{((n+1)(m+1)+1)(nm+1)}{((n+1)m+1)((m+1)n+1)} - \frac{1}{n(n+1)m(m+1)} , \]
we thus have $R_{n,m}=O(n^{-3}m^{-3})$, and we can write
\begin{equation}\label{It1t2}
I(t_1, t_2) =\sum_{n,m=1}^{\infty} \frac{e^{i(t_1 m + t_2 n)}-1}{n(n+1)m(m+1)} + \sum_{n,m=1}^{\infty} \left( e^{i (t_1 m + t_2 n)} -1 \right) R_{n,m} .
\end{equation}

The second term is estimated as
\[ \begin{split} \sum_{n,m=1}^{\infty} \left( e^{i (t_1 m + t_2 n)} -1 \right) R_{n,m} &= \sum_{m=1}^{1/|t_1|} \sum_{n=1}^{1/|t_2|} \left( i t_1 m + i t_2 n + O \left( |t_1 m + t_2 n|^2 \right) \right) R_{n,m} + O \left( t_1^2 + t_2^2 \right) \\ &= i t_1 \sum_{n,m=1}^{\infty} m R_{n,m} + i t_2 \sum_{n,m=1}^{\infty} n R_{n,m} + O \left( t_1^2 \log \frac{1}{|t_1|} + t_2^2 \log \frac{1}{|t_2|} \right) . \end{split} \]
The infinite series is easily computed using telescoping sums:
\[ \begin{split} \sum_{n,m=1}^{\infty} n R_{n,m} &= \sum_{n=1}^{\infty} \left( n \log \frac{(n+1)^2}{n(n+2)} - \frac{1}{n+1} \right) \\ &= \lim_{N \to \infty} \left( \log (N+1) + N \log \frac{N+1}{N+2} - \sum_{n=1}^{N} \frac{1}{n+1} \right) = - \gamma . \end{split} \]
By symmetry, we also have $\sum_{n,m=1}^{\infty} m R_{n,m}=- \gamma$, thus the second term in \eqref{It1t2} is
\begin{equation}\label{It1t2second}
\sum_{n,m=1}^{\infty} \left( e^{i (t_1 m + t_2 n)} -1 \right) R_{n,m} = -i \gamma t_1 -i \gamma t_2 + O \left( t_1^2 \log \frac{1}{|t_1|} + t_2^2 \log \frac{1}{|t_2|} \right) .
\end{equation}

We can rewrite the first term in \eqref{It1t2} as
\[ \sum_{n,m=1}^{\infty} \frac{e^{i(t_1 m + t_2 n)}-1}{n(n+1)m(m+1)} = \left( \sum_{m=1}^{\infty} \frac{e^{it_1 m}}{m(m+1)} \right) \left( \sum_{n=1}^{\infty} \frac{e^{it_2 n}}{n(n+1)} \right) - 1 . \]
Observe that
\[ \sum_{n=1}^{\infty} \frac{z^n}{n(n+1)} = 1+ \frac{1-z}{z} \log (1-z), \qquad |z| \le 1 \]
with the principal branch of the logarithm. For $j=1,2$,
\[ \log (1-e^{i t_j}) = \log |2 \sin (t_j/2)| + i \left( \frac{t_j}{2} - \mathrm{sgn} (t_j) \frac{\pi}{2} \right) = \log |t_j| -i \mathrm{sgn} (t_j) \frac{\pi}{2} + O \left( |t_j| \right) , \]
hence
\[ \sum_{n=1}^{\infty} \frac{e^{it_j n}}{n(n+1)} = 1+ (e^{-it_j}-1) \log (1-e^{it_j}) = 1-i t_j \log |t_j| - \frac{\pi}{2} |t_j| + O \left( t_j^2 \log \frac{1}{|t_j|} \right) . \]
Therefore the first term in \eqref{It1t2} is
\[ \begin{split} \sum_{n,m=1}^{\infty} \frac{e^{i(t_1 m + t_2 n)}-1}{n(n+1)m(m+1)} = &-it_1 \log |t_1| - \frac{\pi}{2} |t_1| - i t_2 \log |t_2| - \frac{\pi}{2} |t_2| \\ &+ O \left( t_1^2 \log \frac{1}{|t_1|} + t_2^2 \log \frac{1}{|t_2|} + |t_1 t_2| \log \frac{1}{|t_1|} \log \frac{1}{|t_2|} \right) . \end{split} \]
The previous formula together with \eqref{It1t2} and \eqref{It1t2second} lead to the claim of the lemma.
\end{proof}

\subsection{Random rationals}

\begin{proof}[Proof of Theorem \ref{Jptheorem}] Let $a/q \sim \mathrm{Unif} (F_Q)$, and consider its continued fraction expansion $a/q=[0;a_1,a_2,\ldots, a_L]$. Then $T^{2j} (a/q) = [0;a_{2j+1}, a_{2j+2}, \ldots, a_L]$. Given $0<p \le \infty$ and $-\infty \le p' <0$, by the definition \eqref{gp} of $g_p$ we can write
\begin{equation}\label{logJp1logJp2}
\begin{split} \left( \log J_p (a/q), \log J_{p'}(a/q) \right) = &\sum_{j \ge 0} \left( h_p (T^{2j}(a/q)), h_{p'}(T^{2j}(a/q)) \right) \\ = &\sum_{j \ge 0} \left( \frac{a_{2j+2}}{8}, - \frac{a_{2j+1}}{8} \right) + \sum_{j \ge 0} \left( g_p (T^{2j}(a/q)), g_{p'} (T^{2j}(a/q)) \right) . \end{split}
\end{equation}

The main term in \eqref{logJp1logJp2} is the first sum. We find its limit distribution by applying \cite[Theorem 3.1]{BD1} with, in the notation of that paper, $m=2$ and the $\mathbb{R}^2$-valued functions $\phi_1(x)=(0,-\frac{1}{8}\lfloor 1/x \rfloor)$ and $\phi_2(x)=(\frac{1}{8} \lfloor 1/x \rfloor, 0)$ to obtain an estimate for the characteristic function of
\[ \sum_{j \ge 1} \phi_{j \textrm{ mod } 2} (T^{j-1}(a/q)) = \sum_{j \ge 0} \left( \frac{a_{2j+2}}{8}, - \frac{a_{2j+1}}{8} \right). \]
In particular, the theorem states that for any $\varepsilon >0$ there exist small constants $\tau=\tau(\varepsilon)>0$ and $\delta=\delta(\varepsilon)>0$ such that for all $t=(t_1,t_2)$ with $|t|<\tau$,
\[ \mathbb{E} \exp \left( i \left( t_1 \sum_{j \ge 0} \frac{a_{2j+2}}{8} - t_2 \sum_{j \ge 0} \frac{a_{2j+1}}{8} \right) \right) = \exp \left( U(t_1,t_2) \log Q + O \left( |t|^{2-\varepsilon} \log Q + |t|^{1-\varepsilon} +Q^{-\delta} \right) \right) \]
with
\[ U(t_1, t_2) = \frac{6}{\pi^2} \int_0^1 \frac{e^{i(\frac{t_1}{8} \lfloor 1/Tx \rfloor - \frac{t_2}{8} \lfloor 1/x \rfloor )}-1}{1+x} \, \mathrm{d}x \]
and an implied constant depending only on $\varepsilon$. Fix constants $x_1, x_2 \in \mathbb{R}$, and choose $t_1=x_1/(\frac{3}{8 \pi}\log Q)$ and $t_2=x_2/(\frac{3}{8 \pi}\log Q)$. Lemma \ref{It1t2lemma} shows that
\[ \begin{split} U \left( \frac{x_1}{\frac{3}{8 \pi}\log Q}, \frac{x_2}{\frac{3}{8 \pi}\log Q} \right) \log Q = &-|x_1|- i \frac{2 \gamma}{\pi} x_1 - i \frac{2}{\pi} x_1 \log \frac{\pi |x_1|}{3 \log Q} \\ &-|x_2|+i \frac{2 \gamma}{\pi} x_2 + i \frac{2}{\pi} x_2 \log \frac{\pi |x_2|}{3 \log Q} + O \left( \frac{(\log \log Q)^2}{\log Q} \right) . \end{split} \]
After subtracting the appropriate centering term, we thus obtain that the characteristic function
\[ \mathbb{E} \exp \left( i \left( x_1 \frac{\sum_{j \ge 0} \frac{a_{2j+2}}{8} - B_Q}{\frac{3}{8 \pi} \log Q} + x_2 \frac{-\sum_{j \ge 0} \frac{a_{2j+1}}{8} + B_Q}{\frac{3}{8 \pi} \log Q} \right) \right) \]
with
\[ B_Q= \frac{3}{4 \pi^2} \log Q \log \log Q -\frac{3}{4 \pi^2}  \left( \gamma+ \log \frac{\pi}{3} \right) \log Q \]
converges pointwise to $\exp (-|x_1| (1+i \frac{2}{\pi} \mathrm{sgn}(x_1) \log |x_1|)) \exp (-|x_2| (1-i \frac{2}{\pi} \mathrm{sgn}(x_2) \log |x_2|))$, which is the characteristic funcion of $\mathrm{Stab}(1,1) \otimes \mathrm{Stab}(1,-1)$. In particular, the first sum in \eqref{logJp1logJp2} satisfies
\begin{equation}\label{ajindistribution}
\left( \frac{\sum_{j \ge 0} \frac{a_{2j+2}}{8} - B_Q}{\frac{3}{8 \pi} \log Q} , \frac{-\sum_{j \ge 0} \frac{a_{2j+1}}{8} + B_Q}{\frac{3}{8 \pi} \log Q} \right) \overset{d}{\to} \mathrm{Stab}(1,1) \otimes \mathrm{Stab}(1,-1) \qquad \textrm{as } Q \to \infty .
\end{equation}

Consider the second sum in \eqref{logJp1logJp2}. Instead of Lemma \ref{It1t2lemma}, we can now use the fact that for any $f\in L^1 ([0,1])$,
\[ \int_0^1 \frac{e^{itf(x)}-1}{1+x} \, \mathrm{d}x = it \int_0^1 \frac{f(x)}{1+x} \, \mathrm{d}x + o(|t|) \qquad \textrm{as } t \to 0. \]
Fix $\varepsilon >0$, and let $g_p^{\pm}$ be as in Lemma \ref{gplemma}. By another application of \cite[Theorem 3.1]{BD1} with $m=2$, $\phi_1(x)=g_p^{\pm}(x) \mp 2c \log (1/Tx)$ and $\phi_2(x)=\pm 2c \log (1/x)$, we deduce
\[ \frac{\sum_{j \ge 0} g_p^{\pm} (T^{2j} (a/q))}{\log Q} \overset{d}{\to} \frac{6}{\pi^2} \int_0^1 \frac{g_p^{\pm}(x)}{1+x} \, \mathrm{d}x \qquad \textrm{as } Q \to \infty , \]
and letting $\varepsilon \to 0$ leads to
\[ \frac{\sum_{j \ge 0} g_p (T^{2j}(a/q))}{\log Q} \overset{d}{\to} \frac{6}{\pi^2} \int_0^1 \frac{h_p (x) - \frac{1}{8} \lfloor 1/Tx \rfloor}{1+x} \, \mathrm{d}x \qquad \textrm{as } Q \to \infty . \]
From \cite[Theorem 3.1]{BD1} with $m=2$, $\phi_1(x)=g_{p'}^{\pm}(x)$ and $\phi_2(x)=0$, we similarly deduce
\[ \frac{\sum_{j \ge 0} g_{p'}(T^{2j}(a/q))}{\log Q} \overset{d}{\to} \frac{6}{\pi^2} \int_0^1 \frac{h_{p'}(x) + \frac{1}{8} \lfloor 1/x \rfloor}{1+x} \, \mathrm{d}x \qquad \textrm{as } Q \to \infty . \]
These formulas combined with \eqref{logJp1logJp2} and \eqref{ajindistribution} immediately yield the joint limit law
\[ \left( \frac{\log J_p (a/q) - E_{p,Q}}{\sigma_Q}, \frac{\log J_{p'}(a/q) - E_{p',Q}}{\sigma_Q} \right) \overset{d}{\to} \mathrm{Stab}(1,1) \otimes \mathrm{Stab} (1,-1) \qquad \textrm{as } Q \to \infty . \]
Since $\log q /\log Q \overset{d}{\to} 1$, we can replace $E_{p,Q}$ by $E_{p,q}$ and $\sigma_Q$ by $\sigma_q$.
\end{proof}

\subsection{Random reals}

Throughout, $\alpha \in [0,1]$ is an irrational number with continued fraction expansion $\alpha=[0;a_1,a_2, \ldots]$ and convergents $p_k/q_k=[0;a_1,a_2,\ldots, a_k]$. Let $\nu (B)=\frac{1}{\log 2}\int_B \frac{1}{1+x} \, \mathrm{d}x$ ($B \subseteq [0,1]$ Borel) denote the Gauss measure on $[0,1]$. The following lemma relies on the classical fact of metric number theory that if $\alpha \sim \nu$, then the sequence of random variables $a_1,a_2,\ldots$ is strictly stationary and $\psi$-mixing with exponential rate. We refer to the monograph \cite{IK} for more context.
\begin{lem}\label{Jpkqklemma} Let $\alpha \sim \nu$. For any $0<p \le \infty$ and $-\infty \le p' <0$,
\[ \left( \frac{\log J_p (p_k/q_k) - A_{p,k}}{\frac{3}{8 \pi} \cdot \frac{\pi^2}{12 \log 2} k}, \frac{\log J_{p'} (p_k/q_k) - A_{p',k}}{\frac{3}{8 \pi} \cdot \frac{\pi^2}{12 \log 2} k} \right) \overset{d}{\to} \mathrm{Stab} (1,1) \otimes \mathrm{Stab}(1,-1) \qquad \textrm{as } k \to \infty , \]
where, for all $p \neq 0$, $A_{p,k}=\mathrm{sgn}(p) \frac{3}{4 \pi^2} \cdot \frac{\pi^2}{12 \log 2} k \log \left( \frac{\pi^2}{12 \log 2} k \right) + D_p \frac{\pi^2}{12 \log 2} k$, with $D_p$ defined in \eqref{Dp}.
\end{lem}

\begin{proof} For the sake of simplicity, we assume that $k$ is even, in which case
\begin{equation}\label{Jp1Jp2formula}
\left( \log J_p (p_k/q_k), \log J_{p'} (p_k/q_k) \right) = \sum_{0\le j<k/2} \left( \frac{a_{2j+2}}{8}, - \frac{a_{2j+1}}{8} \right) + \sum_{0 \le j<k/2} \left( g_p (T^{2j} (p_k/q_k)), g_{p'} (T^{2j} (p_k/q_k) ) \right) . 
\end{equation}
A similar formula holds for odd $k$, the only difference being that the last term in the first sum is $(0,-a_k/8)$, which is negligible in measure.

The main term in \eqref{Jp1Jp2formula} is the first sum, whose limit distribution is easily found using the theory of $\psi$-mixing random variables. Fix real constants $x_1,x_2$ such that $(x_1,x_2) \neq (0,0)$; in what follows, implied constants are allowed to depend on $x_1,x_2$. The random variables
\[ X_j := x_1 \frac{a_{2j+2}/8}{\frac{3}{8 \pi} \cdot \frac{\pi^2}{12 \log 2} k} + x_2 \frac{-a_{2j+1}/8}{\frac{3}{8 \pi} \cdot \frac{\pi^2}{12 \log 2} k}, \qquad 0 \le j <k/2 \]
are identically distributed and $\psi$-mixing with exponential rate. Using the facts that $|e^{iX_j}-1| \le \min \{ |X_j|,2 \}$ and
\[ 1-\cos X_j = 2 \sin^2 (X_j/2) \ge \frac{2}{\pi^2} X_j^2 \mathds{1}_{\{ |X_j| \le \pi \}} , \]
one readily checks that $\mathbb{E}|e^{iX_j}-1| \ll (\log k)/k$ and $\mathbb{E} (1-\cos X_j) \gg 1/k$. Applying \cite[Lemma 1]{HL} with, in the notation of that paper, $P \approx \sqrt{k/ \log k}$ and $m \approx \sqrt{k/\log k}$ yields
\[ \mathbb{E} \exp \left( i \sum_{0 \le j<k/2} X_j \right) = \exp \left( \sum_{0 \le j<k/2} \mathbb{E} \left( e^{i X_j} -1 \right) \right) + O \left( \frac{(\log k)^2}{k} \right) . \]
Lemma \ref{It1t2lemma} with $t_1=x_1 \frac{4 \log 2}{\pi k}$ and $t_2=-x_2\frac{4 \log 2}{\pi k}$ gives that here
\[ \begin{split} \sum_{0 \le j<k/2} \mathbb{E} \left( e^{i X_j} -1 \right) = &\frac{k}{2 \log 2} \int_0^1 \frac{e^{i(t_1 \lfloor 1/Tx \rfloor + t_2 \lfloor 1/x \rfloor)}-1}{1+x} \, \mathrm{d}x \\ = &- |x_1|-i \frac{2 \gamma}{\pi} x_1 - i \frac{2}{\pi} x_1 \log \frac{4 (\log 2) |x_1|}{\pi k} \\ &- |x_2|+ i \frac{2 \gamma}{\pi} x_2 + i \frac{2}{\pi} x_2 \log \frac{4 (\log 2) |x_2|}{\pi k} + O \left( \frac{(\log k)^2}{k} \right) . \end{split} \]
After subtracting the appropriate centering term, we thus obtain that the characteristic function
\[ \mathbb{E} \exp \left( i \left( x_1 \frac{\sum_{0 \le j <k/2} a_{2j+2}/8 -B_k}{\frac{3}{8 \pi} \cdot \frac{\pi^2}{12 \log 2} k} + x_2 \frac{-\sum_{0 \le j<k/2} a_{2j+1}/8 +B_k}{\frac{3}{8 \pi} \cdot \frac{\pi^2}{12 \log 2}k} \right) \right) \]
with
\[ B_k= \frac{3}{4 \pi^2} \cdot \frac{\pi^2}{12 \log 2} k \log \left( \frac{\pi^2}{12 \log 2} k \right) - \frac{3}{4 \pi^2} \left( \gamma + \log \frac{\pi}{3} \right) \frac{\pi^2}{12 \log 2} k \]
converges pointwise to $\exp (-|x_1| (1+i \frac{2}{\pi} \mathrm{sgn}(x_1) \log |x_1|)) \exp (-|x_2| (1-i \frac{2}{\pi} \mathrm{sgn}(x_2) \log |x_2|))$, which is the characteristic funcion of $\mathrm{Stab}(1,1) \otimes \mathrm{Stab}(1,-1)$. In particular, the first sum in \eqref{Jp1Jp2formula} satisfies
\begin{equation}\label{firstsumindistribution}
\left( \frac{\sum_{0 \le j <k/2} a_{2j+2}/8 -B_k}{\frac{3}{8 \pi} \cdot \frac{\pi^2}{12 \log 2} k}, \frac{-\sum_{0 \le j<k/2} a_{2j+1}/8 +B_k}{\frac{3}{8 \pi} \cdot \frac{\pi^2}{12 \log 2}k} \right) \overset{d}{\to} \mathrm{Stab}(1,1) \otimes \mathrm{Stab}(1,-1) \qquad \textrm{as } k \to \infty .
\end{equation}

Consider now the second sum in \eqref{Jp1Jp2formula}. Recall that the Gauss map $T$ is mixing in the sense of ergodic theory, therefore $T^2$ is ergodic. Fix $\varepsilon>0$, and let $g_p^{\pm}$ be as in Lemma \ref{gplemma}. Since $T^{2j}(p_k/q_k)=[0;a_{2j+1}, a_{2j+2}, \ldots, a_k]$ and $T^{2j} \alpha = [0;a_{2j+1}, a_{2j+2}, \ldots ]$, by construction we have
\[ |g_p^{\pm}(T^{2j} (p_k/q_k)) - g_p^{\pm}(T^{2j}\alpha)| \ll |\log [0;a_{2j+2},a_{2j+3},\ldots, a_k] - \log [0;a_{2j+2},a_{2j+3},\ldots ] | . \]
This decays exponentially fast in $k-2j$, hence $\sum_{0 \le j <k/2} g_p^{\pm}(T^{2j} (p_k/q_k)) = \sum_{0 \le j <k/2} g_p^{\pm}(T^{2j} \alpha) +O(1)$. Applying Birkhoff's pointwise ergodic theorem to $T^2$ thus yields
\[ \frac{1}{k/2} \sum_{0 \le j <k/2} g_p^{\pm} (T^{2j} (p_k/q_k)) \to \frac{1}{\log 2} \int_0^1 \frac{g_p^{\pm}(x)}{1+x} \, \mathrm{d} x \qquad \textrm{for a.e. } \alpha, \]
and after letting $\varepsilon \to 0$,
\begin{equation}\label{psipae}
\frac{\sum_{0 \le j <k/2}  g_p (T^{2j} (p_k/q_k))}{\frac{3}{8 \pi} \cdot \frac{\pi^2}{12 \log 2} k} \to \frac{1}{\frac{3}{8 \pi}} \cdot \frac{6}{\pi^2} \int_0^1 \frac{h_p(x) - \frac{1}{8} \lfloor 1/Tx \rfloor}{1+x} \, \mathrm{d}x \qquad \textrm{for a.e. } \alpha .
\end{equation}
We similarly obtain
\[ \frac{\sum_{0 \le j <k/2} g_{p'} (T^{2j} (p_k/q_k))}{\frac{3}{8 \pi} \cdot \frac{\pi^2}{12 \log 2} k} \to \frac{1}{\frac{3}{8 \pi}} \cdot \frac{6}{\pi^2} \int_0^1 \frac{h_{p'}(x) + \frac{1}{8} \lfloor 1/x \rfloor}{1+x} \, \mathrm{d}x \qquad \textrm{for a.e. } \alpha . \]
The previous two relations imply convergence in distribution, and the desired limit law follows from \eqref{Jp1Jp2formula} and \eqref{firstsumindistribution}.
\end{proof}

\begin{proof}[Proof of Theorem \ref{JpMtheorem}] First, let $\alpha \in [0,1]$ be fixed. Recall from \eqref{SNalpha-SNpkqk} that $|S_N(\alpha) -S_N(p_k/q_k)| \ll 1$ for all $0 \le N<q_k$ with a universal implied constant. Therefore if $q_k \le M \le q_K$, then $J_{p,q_k}(\alpha) \le J_{p,M} (\alpha) \le J_{p,q_K}(\alpha)$, and consequently
\begin{equation}\label{JalphaJrk}
\log J_p (p_k/q_k) - O(1) \le \log J_{p,M}(\alpha) \le \log J_p (p_K/q_K) +O(1)
\end{equation}
with universal implied constants. The reverse inequalities hold with $p'$ instead of $p$.

Now let $\alpha \sim \mu$ with a Borel probability measure $\mu$ on $[0,1]$ which is absolutely continuous with respect to the Lebesgue measure. Let $k_M^*=k_M^*(\alpha)$ be the positive integer for which $q_{k_M^*} \le M < q_{k_M^* +1}$. The convergent denominators of Lebesgue-a.e.\ $\alpha$ (and consequently, $\mu$-a.e.\ $\alpha$) satisfy the law of the iterated logarithm
\[ \limsup_{k \to \infty} \frac{\left| \log q_k - \frac{\pi^2}{12 \log 2}k \right|}{\sqrt{k \log \log k}} =c \]
with a universal constant $c>0$; in fact, the central limit theorem also holds for $\alpha \sim \mu$ \cite[Section 3.2.3]{IK}. Therefore
\[ \log q_{k_M^*} - O \left( \sqrt{\log M \log \log \log M} \right) \le \log q_{\lfloor \frac{12 \log 2}{\pi^2} \log M \rfloor} \le \log q_{k_M^*+1} + O \left( \sqrt{\log M \log \log \log M} \right) , \]
and by the general fact $q_{k+2}/q_k \ge 2$ for all $k \ge 1$,
\[ k_M^* = \frac{12 \log 2}{\pi^2} \log M + O \left( \sqrt{\log M \log \log \log M} \right) \qquad \textrm{for $\mu$-a.e. } \alpha . \]
Letting $k_M$ be the even integer closest to, say, $\frac{12 \log 2}{\pi^2} \log M - (\log M)^{3/4}$ and $K_M$ be the even integer closest to, say, $\frac{12 \log 2}{\pi^2} \log M + (\log M)^{3/4}$, we thus have $\mu (\{ \alpha \in [0,1] \, : \, k_M \le k_M^* \le K_M \})=1-o(1)$ as $M \to \infty$. By \eqref{JalphaJrk}, we can write $\log J_{p,M} (\alpha) = \log J_p (p_{k_M}/q_{k_M}) + \xi_{p,M} (\alpha)$, with an error term $\xi_{p,M}(\alpha)$ which outside a set of $\mu$-measure $o(1)$ satisfies
\[ |\xi_{p,M}(\alpha)| \ll 1+ \left| \log J_p (p_{K_M}/q_{K_M}) - \log J_p (p_{k_M}/q_{k_M}) \right| \]
with a universal implied constant. The same holds with $p'$ instead of $p$.

Recall the decomposition formula \eqref{Jp1Jp2formula} for $(\log J_p (p_{k_M}/q_{k_M}), \log J_{p'} (p_{k_M}/q_{k_M}))$. According to Lemma \ref{Jpkqklemma}, if $\alpha \sim \nu$, then
\[ \left( \frac{\log J_p (p_{k_M}/q_{k_M}) - E_{p,M}}{\sigma_M}, \frac{\log J_{p'} (p_{k_M}/q_{k_M}) - E_{p',M}}{\sigma_M} \right) \overset{d}{\to} \mathrm{Stab}(1,1) \otimes \mathrm{Stab} (1,-1) \qquad \textrm{as } M \to \infty . \]
In fact, the same holds if $\alpha \sim \mu$. Indeed, this easily follows from a mixing property of the Gauss map \cite[p.\ 166]{IK}
\[ \lim_{n \to \infty} \sup_{A \in \mathcal{F}_n^{\infty}} \left| \mu (A) - \nu (A) \right| =0, \]
where $\mathcal{F}_n^{\infty}$ denotes the $\sigma$-algebra generated by the partial quotients $a_m$, $m \ge n$. Note that the terms $j \ge n/2$ in \eqref{Jp1Jp2formula} are $\mathcal{F}_n^{\infty}$-measurable.

It remains to show that $\xi_{p,M}(\alpha)=o(\log M)$ and $\xi_{p',M}(\alpha) = o(\log M)$ in $\mu$-measure. By the decomposition formula \eqref{Jp1Jp2formula},
\[ \begin{split} | \log J_p (p_{K_M}/&q_{K_M}) - \log J_p (p_{k_M}/q_{k_M}) | \le \\ &\sum_{k_M/2 \le j< K_M/2 } \frac{a_{2j+2}}{8} + \left| \sum_{0 \le j <K_M/2} g_p (T^{2j}p_{K_M} / q_{K_M}) - \sum_{0 \le j <k_M/2} g_p (T^{2j}p_{k_M} / q_{k_M}) \right| . \end{split} \]
Recall that for any $j \ge 1$ and any real $t \ge 1$,
\[ \nu \left( \left\{ \alpha \in [0,1] \, : \, a_j \ge t \right\} \right) = \frac{1}{\log 2} \sum_{n \ge t} \log \left( 1+\frac{1}{n(n+2)} \right) \ll \frac{1}{t} . \]
Since $K_M-k_M \ll (\log M)^{3/4}$, the union bound thus yields
\[ \nu \bigg( \bigg\{ \alpha \in [0,1] \, : \, \sum_{k_M/2 \le j < K_M/2} a_{2j+2} \ge \varepsilon \log M \bigg\} \bigg) \ll \frac{1}{\varepsilon (\log M)^{1/4}} . \]
In particular, $\sum_{k_M/2 \le j < K_M/2} a_{2j+2} =o(\log M)$ in $\nu$-measure, and consequently also in $\mu$-measure. Formula \eqref{psipae} shows that
\[ \sum_{0 \le j <K_M/2} g_p (T^{2j}p_{K_M} / q_{K_M}) - \sum_{0 \le j <k_M/2} g_p (T^{2j}p_{k_M} / q_{k_M}) = o(\log M) \]
holds for Lebesgue-a.e.\ $\alpha$, and consequently also in $\mu$-measure. This finishes the proof of $\xi_{p,M}(\alpha) = o(\log M)$ in $\mu$-measure, and the same arguments show that this holds with $p'$ instead of $p$ as well.
\end{proof}

\begin{proof}[Proof of Theorem \ref{JpMtildetheorem}] This is entirely analogous to the proof of Theorem \ref{JpMtheorem}. The only difference is that instead of $|S_N(\alpha) - S_N(p_k/q_k)| \ll 1$, we use $|\tilde{S}_N (\alpha) - \tilde{S}_N (p_k/q_k)| \ll \max_{1 \le j \le k} \log (a_j+1)$, see \cite[Proposition 3]{AB2}. In particular, $|\tilde{S}_N (\alpha) - \tilde{S}_N (p_k/q_k)| \ll \log (k+1)$ for Lebesgue-a.e.\ $\alpha$, which suffices for our purposes.
\end{proof}

\section*{Acknowledgments} The author is supported by the Austrian Science Fund (FWF) project M 3260-N.

\end{document}